\documentclass[a4paper,twoside,12pt,leqno]{article}
\usepackage{amsmath,amsthm,amssymb,enumerate}
\usepackage{eucal}
\usepackage{epsfig, float, afterpage, array}
\usepackage{hhline}
\usepackage{graphpap}

\swapnumbers
\theoremstyle{plain}
\newtheorem{theorem}{Theorem}[section]
\newtheorem{lemma}[theorem]{Lemma}

\newtheorem{proposition}[theorem]{Proposition}

\theoremstyle{definition}
\newtheorem{definition}[theorem]{Definition}

\newtheorem{definitions}[theorem]{Definitions}
\newtheorem{example}[theorem]{Example}

\newtheorem{notation}[theorem]{Notation}

\newtheorem{remark}[theorem]{Remark}
\newtheorem{remarks}[theorem]{Remarks}

\newtheorem{review}[theorem]{Review}
\numberwithin{equation}{theorem}
\numberwithin{figure}{theorem}

{\addvspace{\baselineskip}}

\newcommand{\abs}[1]{\vert#1\vert} 

\newcommand{\gen}[1]{\langle#1\rangle}

\DeclareMathOperator{\Out}{Out}

\DeclareMathOperator{\Aut}{Aut}

\DeclareMathOperator{\Hom}{Hom}

\def \M {\mathcal{M}}
\def \AM {\mathcal{AM}}
\def \G {\mathcal{G}}
\def \naturals{\mathbb{N}}
\def \integers {\mathbb{Z}}
\def \reals{\mathbb{R}}
\def \complexes{\mathbb{C}}

\def\d1{\discretionary{-}{}{-}}

\def\coloneq{{\colon}\!\!\!\!=}

\begin{document}

\title{Embedding mapping-class groups of orientable surfaces with one boundary component}

\author
{Llu\'is Bacardit\footnote{The research was
funded by Conseil R\'{e}gional de Bourgogne and the MIC (Spain) through
Project MTM2008-01550.} }

\date{}

\maketitle


 \bigskip

\begin{abstract} Let $S_{g,1,p}$ be an orientable surface of genus $g$ with one boundary component and $p$ punctures. Let $\M_{g,1,p}$ be the mapping-class group of $S_{g,1,p}$ relative to the boundary. We construct homomorphisms $\M_{g,1,p} \to \M_{g',1,(b-1)}$, where $g' \geq 0$ and $b\geq 1$. We proof that the constructed homomorphisms $\M_{g,1,p} \to \M_{g',1,(b-1)}$ are injective. One of these embeddings for $g = 0$ is classic.
\bigskip

{\footnotesize
\noindent \emph{{\normalfont 2000}\,Mathematics Subject Classification.} Primary: 20F34;
Secondary: 20E05, 20E36, 57M99.

}
\end{abstract}

\section{General Notation}

Let $\naturals$ denote the set of finite cardinals, $\{0,1,2,\ldots\}$.\medskip

Throughout, we fix elements $g,p$ of $\naturals$.\medskip

Given two sets $A$ and $B$, we denote by $A\vee B$ the disjoint union of $A$ and $B$.\medskip

Let $G$ be a multiplicative group. For elements $a$, $b$ of  $G$,  we write $\overline a \coloneq a^{-1}$,   $a^b \coloneq \overline b a b$,
 $[a]\coloneq \{a^g\mid g \in G\}$,  the conjugacy class of $a$ in $G$.
We let~$\Aut(G)$ denote the group of all automorphisms of $G$,
acting on $G$ on the right with exponent notation.\medskip

An {\it ordering} of a set will mean a {\it total} ordering for the set.\medskip

We will make frequent use of sequences, usually with vector notation.
We shall use the language of sequences to introduce
indexed symbols and to realize free monoids.
Formally, we define a  {\it sequence} as a set endowed with a specified listing of its elements.
Thus a sequence has an underlying set; with vector notation, the coordinates are the elements of
(the underlying set of) the sequence.
For two sequences $A$, $B$,  their  concatenation will be denoted~$A \vee B$.
By a sequence $A$ {\it in} a given set $X$,
we mean a sequence endowed with a specified map of sets $A \to X$; to avoid extra notation,
we shall use the same symbol to denote an element of $A$ and its image in $X$ even
when the map is not injective. We often treat $A$ as an element in the free monoid on $X$
with concatenation as binary operation, and then the elements of $A$ are its atomic factors.

Let $i$, $j \in \integers$. We write 

$[i{\uparrow}j]\coloneq
 \begin{cases}
(i,i+1,\ldots, j-1, j) \in \integers^{j-i+1}   &\text{if $i \le j$,}\\
() \in \integers^0 &\text{if $i > j$.}
\end{cases}
$ \newline
Also, $[i{\uparrow}\infty[\,\, \coloneq (i,i+1,i+2,\ldots)$.
We define $[j{\downarrow}i]$ to be the reverse of the sequence $[i{\uparrow}j]$,
$(j,j-1,\ldots,i+1,i)$.

Let $v$ be a symbol. For each $k \in \integers$, we let $v_k$ denote the ordered pair
$(v,k)$. We let 

$v_{[i{\uparrow}j]}\coloneq  \begin{cases}
(v_i,v_{i+1},   \cdots, v_{j-1},   v_j)  &\text{if $i \le j$,}\\
()  &\text{if $i > j$.}
\end{cases}$ \newline
Also, $v_{[i{\uparrow}\infty[\,} \coloneq (v_i,v_{i+1},v_{i+2},\ldots)$.
We define $v_{[j{\downarrow}i]}$ to be the reverse of the sequence~$v_{[i{\uparrow}j]}$.

Suppose there is specified a set-map $v_{[i{\uparrow}j]} \to A$. We treat the
elements of $v_{[i{\uparrow}j]}$ as elements of $A$ (possibly with repetitions), and, we say 
that $v_{[i{\uparrow}j]}$ is a {\it sequence in} $A$.

If $v_{[i{\uparrow}j]}$ is a sequence in a multiplicative group $G$, we let
\begin{align*}\Pi v_{[i{\uparrow}j]}&\coloneq  \begin{cases}
v_i   v_{i+1}   \cdots v_{j-1}   v_j \in G   &\text{if $i \le j$,}\\
1 \in G &\text{if $i > j$.}
\end{cases}\\ \Pi v_{[j{\downarrow}i]} &\coloneq  \begin{cases}
v_j   v_{j-1}   \cdots v_{i+1}   v_i \in G   &\text{if $j \ge i$,}\\
1 \in G &\text{if $j< i $.}
\end{cases}\end{align*}

Let $F$ be the free group on a set $X$. 
Consider an element $w$ of $F$ and a sequence $a_{[1{\uparrow}k]}$ in
$X \vee \overline X$.  
If $\Pi a_{[1{\uparrow}k]} = w$ in $F$,
we say that  $a_{[1{\uparrow}k]}$ is a monoid {\it expression} for~$w$ in
$X\vee \overline X$ of {\it length} $k$.
We say that  $a_{[1{\uparrow}k]}$ is  {\it reduced} if, for all $j \in [1{\uparrow}(k-1)]$,
$a_{j+1}\ne \overline a_j$ in $X\vee\overline X$.
Each element of $F$ has a unique reduced expression, called the {\it normal form}. Suppose that  $a_{[1{\uparrow}k]}$
is the normal form for $w$.  We define the {\it length} of $w$ to be $\abs{w}\coloneq k$.

For $p$ an element of $\naturals\cup \{\infty\},\, p\neq 0$, let $C_p$ be the {\it cyclic group of order} $p$ with multiplicative notation. For $q\in \naturals$, let $C_{p}^{\ast q}$ denote the free product of $q$ copies of $C_p$.

\section{Introduction and main results}
\label{sec:intro}

Recall $g,p$ are elements of $\naturals$. Let $b$ be an element of $\naturals$, $b\geq 1$. Let $S_{g,b,p}$ be an orientable surface of genus $g$ with $b$ boundary components and $p$ punctures.

Let $\Sigma_{g,b,p}$ be the rank $2g+b-1+p$ free group with generating set $x_{[1{\uparrow}g]}\vee y_{[1{\uparrow}g]}\vee z_{[1{\uparrow}(b-1)]}\vee t_{[1{\uparrow}p]}$. We view $\Sigma_{g,b,p}$ as a presentation of $\pi_1(S_{g,b,p}, \ast)$, the fundamental group of $S_{g,b,p}$ based at a point $\ast$ in the $b$-th boundary component. In addition, for every $l \in [1{\uparrow}(b-1)]$, $z_l$ represents a loop around the $l$-th boundary component; for every $k \in [1{\uparrow}p]$, $t_k$ represents a loop around the $k$-th puncture, and $(\Pi_{i\in [1{\uparrow}g]} [x_i, y_i] \Pi z_{[1{\uparrow}(b-1)]}\Pi t_{[1{\uparrow}p]})^{-1}$ represents a loop around the $b$-th boundary component. Note that, if $p = 0$, there is no puncture in $S_{g,b,p} = S_{g,b,0}$, $t_{[1{\uparrow}p]}$ is the empty sequence, and $\Pi t_{[1{\uparrow}p]} = 1$.

Let $\AM_{g,b,p}$ denote the subgroup of $\Aut(\Sigma_{g,b,p} \ast \langle e_{[1{\uparrow}(b-1)]} \mid \,\rangle)$ consisting of all the automorphisms of $\Sigma_{g,b,p} \ast \langle e_{[1{\uparrow}(b-1)]} \mid \,\rangle$ which map $\Sigma_{g,b,p}$ to itself and respect the sets $$\{\Pi_{i\in [1{\uparrow}g]} [x_i, y_i] \Pi z_{[1{\uparrow}(b-1)]} \Pi t_{[1{\uparrow}p]}\},\,\, \{\overline z_1^{\,e_1}\},\, \{\overline z_2^{\,e_2}\}, \ldots  ,\, \{\overline z_{(b-1)}^{\,e_{(b-1)}}\},\,\, \{[\,\overline t_k]\}_{k\in [1{\uparrow}p]}.$$

We call $\AM_{g,b,p}$ the {\it algebraic mapping-class group} of 
a surface of genus $g$ with $b$ boundary components and $p$ punctures, $S_{g,b,p}$.\medskip

For $b = 0$, the (orientation-preserving) mapping-class group of $S_{g,0,p}$, denoted $\M_{g,0,p}$, is defined as the group of orientation-preserving homeomorphisms of $S_{g,0,p}$ modulo isotopy. Let $f$ be a homeomorphism of $S_{g,0,b}$, then $f$ induces an automorphism of $\pi_1(S_{g,0,p})$ which respects the set of conjugacy classes of $t_{[1{\uparrow}p]}$. Since $f$ is not forced to fix the base point of $\pi_1(S_{g,0,p})$, the isotopy class of $f$ defines an automorphism of $\pi_1(S_{g,0,p})$ up to conjugation, hence an element of $\Out(\pi_1(S_{g,0,p}))$. By the Dehn-Nielsen-Baer theorem, this correspondence is an isomorphism onto the subgroup of $\Out(\pi_1(S_{g,0,p}))$ which respects the set of conjugacy classes of $t_{[1{\uparrow}p]}$, \cite[Theorem 3.6]{FM}, \cite[Theorem 2.9.A]{Ivanov}. In particular, $\M_{g,0,p} \leq \Out(\pi_1(S))$.

If $b \geq 1$, that is, $S_{g,b,p}$ has non-empty boundary; we restrict ourselves to homeomorphisms and isotopies of $S_{g,b,p}$ which fix the boundary pointwise. These homeomorphisms preserve the orientation of $S_{g,b,p}$. In this case, we take the base point of $S_{g,b,p}$ in the $b$-th boundary component of $S_{g,b,p}$. We convert a boundary component of $S_{g,b,p}$ into a puncture by identifying via a homeomorphism the boundary component with the boundary of a once punctured disc. If $b\geq 1$, by converting all the boundary components into punctures we can deduce $\M_{g,b,p} \simeq \AM_{g,b,p}$ from the Dehn-Nielsen-Baer theorem, \cite[Theorem 9.6]{DF05}.
See~\cite{DF05} for a background on algebraic mapping-class groups, with some changes of notation. From now on, we will deal with $\AM_{g,b,p}$ and, mostly, in the case $b=1$.\medskip

For $p \geq 1$ and $q\in [1{\uparrow}p]$, we denote by $\AM_{g,b,q\perp (p-q)}$ the subgroup of $\AM_{g,1,p}$ consisting of all the automorphisms which respect the sets $\{[\,\overline t_k]\}_{k\in[1{\uparrow}q]}$ and $\{[\,\overline t_k]\}_{k\in[(q+1){\uparrow}p]}$. Let $N$ be the normal closure in $\Sigma_{g,b,p}$ of $t_{[(q+1){\uparrow}p]}$. Notice $\Sigma_{g,b,q} \simeq \Sigma_{g,b,p} \slash N$. We say that we have eliminated the last $(p-q)$ punctures. Since $N$ is $\AM_{g,b,q\perp (p-q)}$-invariant, we can define a homomorphism $\AM_{g,b,q\perp (p-q)} \to \AM_{g,b,q}$ which eliminates the last $(p-q)$ punctures, see \cite[Section 11]{DF05}.\medskip

Suppose $b\geq 2$. Let $\phi\in \AM_{g,b,p}$ and $l\in [1{\uparrow}(b-1)]$. Since $(\overline z_l^{\,e_l})^\phi = \overline z_{l}^{\,e_{l}}$, we see $\overline z_l^{\,\phi} = \overline z_{l}^{\,w_l}$ for some $w_l \in \Sigma_{g,b,p} \ast \langle e_{[1{\uparrow}(b-1)]} \mid \,\rangle$. Since $\phi$ maps $\Sigma_{g,b,p}$ to itself and $z_l \in \Sigma_{g,b,p}$, we see $w_l \in \Sigma_{g,b,p}$. Since $\overline e_{l} \overline z_{l} e_{l} = \overline z_{l}^{\,e_{l}} = (\overline z_l^{\,e_l})^\phi = \overline e_l^{\,\phi}\overline z_l^{\, \phi} e_l^{\phi} = \overline e_l^{\,\phi}(\overline w_l\overline z_{l} w_l) e_l^{\phi}$, we see $(e_l^{\phi}\overline e_{l}) \overline z_{l} (e_{l} \overline e_l^{\phi})= \overline w_l \overline z_{l} w_l$ and $e_l^{\,\phi} \overline e_{l} \in \overline w_l \langle z_{l} \rangle$. Hence, $e_l^{\phi} \in \overline w_l \langle z_{l} \rangle e_l$.

In $\M_{g,b,p}$, the difference between a puncture and a boundary component is that the Dehn twist with respect to a loop around a puncture is trivial in $\M_{g,b,p}$ and the Dehn twist with respect to a loop around a boundary component is not trivial in $\M_{g,b,p}$. In $\AM_{g,b,p}$, this fact is capturated by the fact that for $l\in [1{\uparrow} (b-1)]$ the map 
$$\left\{ 
\begin{array}{lll}
e_l & \mapsto & z_le_l,\\
a  & \mapsto & a, \quad a \in x_{[1{\uparrow}g]} \vee y_{[1{\uparrow}g]} \vee t_{[1{\uparrow}p]} \vee z_{[1{\uparrow}(b-1)]} \vee e_{[1{\uparrow}(l-1)]} \vee e_{[(l+1){\uparrow}(b-1)]}.
\end{array}
\right.$$
defines an element of $\AM_{g,b,p}$. We can see $e_l$ as an arc from the base point in the $b$-th boundary component to a point in the $l$-th boundary component.

Consider the homomorphism $\Sigma_{g,b,p} \ast \langle e_{[1{\uparrow}(b-1)]} \rangle \to \Sigma_{g,b-1,p+1} \ast \langle e_{[1{\uparrow}(b-2)]} \rangle$  such that $e_{(b-1)} \mapsto 1$, $z_{(b-1)} \mapsto t_{p+1}$ and identifies all the other generators. This homomorphism corresponds to converting the $(b-1)$-th boundary component to a puncture. This homomorphism induces a homomorphism $\AM_{g,b,p} \to \AM_{g,b-1,p \perp 1}$ which forgets $e_{b-1}$, see \cite[Section 9]{DF05}.\medskip

For $g = 0, \, b=1$ and $p\geq 1$, $\AM_{0,1,p}$ is isomorphic to the {\it $p$-string braid group}. We have $\AM_{0,1,p} = \langle \sigma_{[1{\uparrow}(p-1)]}\rangle$, where for all $i \in [1{\uparrow}(p-1)]$, $\sigma_i \in \Aut(\Sigma_{0,1,p})$ is defined by
\begin{equation}\label{eq:trena}
\sigma_{i} \coloneq \left\{
\begin{array}{llll}
t_i & \mapsto & t_{i+1},\\
t_{i+1} & \mapsto & t_i^{t_{i+1}},\\
t_k & \mapsto & t_k &\text{if } k \in [1{\uparrow}(i-1)]\vee [(i+2){\uparrow}p].
\end{array}
\right.
\end{equation}\medskip

Let $d\in \naturals,\, d\geq 2$.

Let $\Sigma_{g,1,p^{(d)}}$ denote the group $\langle x_{[1{\uparrow}g]}\vee y_{[1{\uparrow}g]}\vee \tau_{[1{\uparrow}p]}\mid \tau^d_1, \tau^d_2, \ldots, \tau^d_p\rangle$. Hence, $\Sigma_{g,1,p^{(d)}} \simeq \Sigma_{g,1,0} \ast C_d^{\ast p}$. Notice $\Sigma_{g,1,p^{(d)}} = \Sigma_{g,1,p}$, if $p = 0$.\medskip

Let $\AM_{g,1,p^{(d)}}$ denote the group of all automorphisms of $\Sigma_{g,1,p^{(d)}}$ that respect the sets $$\{\Pi_{i\in [1{\uparrow}g]} [x_i, y_i] \Pi \tau_{[1{\uparrow}p]}\},\quad \{[\overline \tau_k]\}_{k\in [1{\uparrow}p]}.$$\medskip

Let $\mathcal P_{[1{\uparrow}p]}$ be a set of $p$ interior points of $S_{g,1,0}$. For each $k\in [1{\uparrow}p]$, let $D(\mathcal P_k)$ be an open disc in the interior of $S_{g,1,0}$ and centered at $\mathcal P_k$. For each $k\in [1{\uparrow}p]$, let $D_k$ be a copy of the closed disc $\{z\in \complexes \mid \abs{z} \leq 1\}$. We define \[S_{g,1,p^{(d)}} \coloneq ((S_{g,1,0} - \vee_{k\in [1{\uparrow}p]} D(\mathcal P_k)) \vee (\vee_{k\in [1{\uparrow}p]}D_k)) \slash \sim\]
where $\sim$ is the following identification. For every $k\in [1{\uparrow}p]$, we identify the boundary of $D(\mathcal P_k)$, denoted $\partial D(\mathcal{P}_k)$, with a copy of $S^1=\{z\in \complexes \mid \abs{z} = 1\}$, and the boundary of $D_k$, denoted $\partial D_k$, with another copy of $S^1$. We define $f_k: \partial D_k \to \partial D(\mathcal{P}_k)$, $z \mapsto z^d$. Now, $\sim$ is defined by identifying $z$ and $f_k(z)$ for every $z \in \partial D_k$ and every $k\in [1{\uparrow}p]$.

It can be seen that $\Sigma_{g,1,p^{(d)}}$ is the fundamental group of $S_{g,1,p^{(d)}}$.

By analogy, $\AM_{g,1,p^{(d)}}$ will be called the {\it (algebraic) mapping class group} of $S_{g,1,p^{(d)}}$.\medskip

Since the elements of $\AM_{g,1,p}$ respect the set $\{[\,\overline t_k]\}_{k\in [1{\uparrow}p]}$, the natural homomorphism $\Sigma_{g,1,p} \to \Sigma_{g,1,p^{(d)}}$ induces a natural homomorphism \[\psi:\AM_{g,1,p} \to \AM_{g,1,p^{(d)}}.\]
If $p = 0$, then $\Sigma_{g,1,p} = \Sigma_{g,1,p^{(d)}}$ and $\psi$ is the identity.

\begin{theorem}\label{thm:principal}
The homomorphism 
$
\psi:\AM_{g,1,p} \to \AM_{g,1,p^{(d)}}
$
 is injective for all $p \in \naturals$.
\end{theorem}
\medskip

Let $\kappa : S_{g',b,0} \to S_{g,1,0}$ be an index $m \in \naturals$ branched regular cover with $p$ branching points in the interior of $S_{g,1,0}$ which lift to $q$ points in $S_{g',b,0}$. Notice that $q = 0$ if and only if $p = 0$. Let $\kappa': S_{g',b,q} \to S_{g,1,p}$ be the corresponding unbranched cover. We identify $\Sigma_{g',b,q} = \pi_1(S_{g',b,q}, \hat \ast)$ with $\kappa'_*(\Sigma_{g',b,q})$, where $\hat \ast$ is a point in the $b$-th boundary component of $S_{g',b,q}$. Notice that $\Sigma_{g',b,q}$ is a normal subgroup of $\Sigma_{g,1,p}$ of index $m$. We set $G \coloneq \Sigma_{g,1,p} \slash \Sigma_{g',b,q}$ the group of deck transformations.

We put $\varrho = \Pi_{i\in [1{\uparrow}g]} [x_i, y_i] \Pi t_{[1{\uparrow}p]} \Sigma_{g',b,q} \in G$. Let $c$ be the order of $\varrho$ in $G$. Since $\varrho^c = 1$ in $G$, we see that $(\Pi_{i\in [1{\uparrow}g]} [x_i, y_i] \Pi t_{[1{\uparrow}p]})^c \in \Sigma_{g',b,q}$. Notice that $(\Pi_{i\in [1{\uparrow}g]} [x_i, y_i] \Pi t_{[1{\uparrow}p]})^{-c}$ represents a loop around the $b$-th boundary component. We take a basis $\hat x_{[1{\uparrow}g']} \vee \hat y_{[1{\uparrow}g']} \vee \hat z_{[1{\uparrow}(b-1)]} \vee \hat t_{[1{\uparrow}q]}$ of $\Sigma_{g',b,q} = \pi_1(S_{g',b,q}, \hat \ast)$ such that $$\Pi_{i\in [1{\uparrow}g']} [\hat x_i, \hat y_i] \Pi \hat z_{[1{\uparrow}(b-1)]} \Pi \hat t_{[1{\uparrow}q]} = (\Pi_{i\in [1{\uparrow}g]} [x_i, y_i] \Pi t_{[1{\uparrow}p]})^c.$$
Recall $G$ has cardinality $m$. The subgroup $\langle \varrho \rangle \leq G$ has index $b=m\slash c$. 
For every $l\in [1{\uparrow}(b-1)]$, let $w_l \in \Sigma_{g,1,p} - \Sigma_{g',b,q}$ such that $$\hat z_l = \overline w_l (\Pi_{i\in [1{\uparrow}g]} [x_i, y_i] \Pi t_{[1{\uparrow}p]})^{-c} w_l.$$
We put $\rho_l = w_l \Sigma_{g',b,q} \in G$. Then $G = \langle \varrho\rangle\rho_1 \vee \langle \varrho\rangle\rho_2 \cdots \vee \langle \varrho\rangle\rho_{(b-1)} \vee \langle \varrho\rangle$. That is, the boundary components of $S_{g',b,p}$ are image by deck transformations of the $b$-th boundary component.

For every $k\in [1{\uparrow}p]$ we put $\varrho_k = t_k \Sigma_{g',b,q} \in G$. Let $d_k$ be the order of $\varrho_k$ in $G$. Since $t_k$ corresponds to a branching point, $t_k \notin \Sigma_{g',b,q}$ and $d_k \geq 2$. Since $\varrho_k^{d_k} = 1$ in $G$, we see that $t_k^{d_k} \in \Sigma_{g',b,q}$. Notice that $t_k^{d_k}$ represents a loop around a lift of the $k$-th puncture of $S_{g,1,p}$. 
The subgroup $\langle \varrho_k \rangle$ has index $m_k = m\slash d_k$ in $G$. 
Hence, $G = \langle \varrho_k\rangle\rho_{1,k} \vee \langle \varrho_k\rangle\rho_{2,k} \cdots \vee \langle \varrho_k\rangle\rho_{m_k,k}$, where $\rho_{i,k} = u_{i,k} \Sigma_{g',b,q} \in G$ for all $i \in [1{\uparrow}m_k]$. Notice that $(t_k^{d_k})^{u_{1,k}}, (t_k^{d_k})^{u_{2,k}}, \ldots, (t_k^{d_k})^{u_{m_k,k}}$ represent loops around the $m_k$ lifts of the $k$-th puncture. We choose $u_{1, k}, u_{2, k}, \ldots, u_{m_k, k}$ such that 
$\{(t_k^{d_k})^{u_{1,k}}, (t_k^{d_k})^{u_{2,k}}, \ldots, (t_k^{d_k})^{u_{m_k,k}}\} \subseteq \{\hat t_1, \hat t_2, \ldots, \hat t_q\}$. Then
\[
\{\hat t_1, \hat t_2, \ldots, \hat t_q\} = \bigvee_{k\in [1{\uparrow}p]} \{(t_k^{d_k})^{u_{1,k}}, (t_k^{d_k})^{u_{2,k}}, \ldots, (t_k^{d_k})^{u_{m_k,k}}\}.
\]

Suppose, now, that $\Sigma_{g',b,q}$ is $\AM_{g,1,p}$-invariant. It is easy to see that $d_1=d_k \geq 2$ for all $k\in [1{\uparrow}p]$. Let $d = d_1$. Every $\phi \in \AM_{g,1,p}$ induces an automorphisms of $\Sigma_{g',b,q}$ by restriction. Hence, we have a homomorphism $\AM_{g,1,p} \to \Aut (\Sigma_{g',b,q})$ given by restriction. Since, in $\Sigma_{g,1,p}$, $\Pi_{i \in [1{\uparrow}g']} [\hat x_i,\hat y_i] \Pi \hat z_{[1{\uparrow}(b-1)]} \Pi \hat t_{[1{\uparrow}q]} = (\Pi_{i\in [1{\uparrow}g]} [x_i,y_i] \Pi t_{[1{\uparrow}p]})^{c}$, $\hat z_{l}$ is conjugate to $(\Pi_{i\in [1{\uparrow}g]} [x_i,y_i] \Pi t_{[1{\uparrow}p]})^{-c}$ for all $l \in [1{\uparrow}(b-1)]$, and $\hat t_{k}$ is conjugate to $t_j^d$, where $j\in [1{\uparrow}p]$, for all $k\in [1{\uparrow}q]$, we have that the image of the homomorphism $\AM_{g,1,p} \to \Aut (\Sigma_{g',b,q})$ lies inside $\AM_{g',1,(b-1)\perp q}$.

Since $\Sigma_{g',b,q}$ is $\AM_{g,1,p}$-invariant, every homeomorphism of $S_{g,1,p}$ lifts to a homeomorphism of $S_{g',b,q}$ which fixes the $b$-th boundary component pointwise, but this lift may not fix the first $(b-1)$ boundary components pointwise. If we convert the first $(b-1)$ boundary components of $S_{g',b,p}$ into punctures, this is not a problem. If we want to conserve the first $(b-1)$ boundary components, we have to restrict ourselves to homeomorphisms of $S_{g,1,p}$ whose lifts fixe the boundaries pointwise. Algebraically, if we want to have a homomorphism inside $\AM_{g',b,q}$ we have to define the image of $\hat e_{[1{\uparrow}(b-1)]}$. To do this, we need to restrict ourselves to the following subgroup of $\AM_{g,1,p}$ (since $\Sigma_{g',b,q}$ is $\AM_{g,1,p}$-invariant, every element of $\AM_{g,1,p}$ induces an automorphism of $G=\Sigma_{g,1,p} \slash \Sigma_{g',b,q}$). $$\AM^G_{g,1,p} \coloneq \{\phi \in \AM_{g,1,p} \mid \phi \textrm{ induces the identity of } G\}.$$

\begin{theorem}\label{thm:inclusio} With the above notation, if $\Sigma_{g',b,q}$ is $\AM_{g,1,p}$-invariant and $(g,p,d) \neq (0,2,2)$ then the composition 
\begin{equation}\label{eq:suc_ex}
\AM_{g,1,p} \to \AM_{g',1,(b-1)\perp q} \to \AM_{g',1,(b-1)}
\end{equation} where the first homomorphism is given by restriction and the second homomorphism is given by eliminating the last $q$ punctures, is injective.\medskip
\end{theorem}

Theorem~\ref{thm:inclusio} gives an algebraic proof of the following theorem, which is an analog for surfaces with one boundary component of a theorem of Birman and Hilden~\cite[Theorem 2]{BirmanHilden} (the hypothesis that $\kappa_*'(\Sigma_{g',b,q})$ is $\AM_{g,1,p}$-invariant can be removed, but then we need extra notation).

\begin{theorem}\label{thm:BH}
Let $\kappa: S_{g',b,0} \to S_{g,1,0}$ be a finite index regular cover with $p$ branching points in $S_{g,1,0}$. Let $\hat f$ be an homeomorphism of $S_{g',b,0}$ which fixes the $b$-th boundary component pointwise and preserves the fibers of $\kappa: S_{g',b,0} \to S_{g,1,0}$. Then $\hat f$ induces an homeomorphism $f$ of $S_{g,1,0}$ such that $\kappa \hat f = f \kappa$. Let $\kappa': S_{g',b,q} \to S_{g,1,p}$ be the corresponding unbranched cover. Suppose $\kappa_*'(\Sigma_{g',b,q})$ is $\AM_{g,1,p}$-invariant and $(g,p) \neq (0,2)$. If $\hat f$ is isotopic to the identity relative to the $b$-th boundary component, then $f$ is isotopic to the identity relative to the boundary.
\end{theorem}

\begin{proof}{\it (using Theorem~\ref{thm:inclusio}).} It is a general fact that if $\hat f$ preserves the fibers of $\kappa: S_{g',b,0} \to S_{g,1,0}$, then $\hat f$ induces an homeomorphism $f$ of $S_{g,1,0}$ such that $\kappa \hat f = f \kappa$. In particular, $f$ sends branching points to branching points. 

Let $\kappa': S_{g',b,q} \to S_{g,1,p}$ be the corresponding unbranched cover. We identify $\Sigma_{g',b,q}$ and $\kappa'_*(\Sigma_{g',b,q})$.
Since $f$ sends branching points to branching points, $f$ restricts to a homeomorphism $g$ of $S_{g,1,p}$. Since $\Sigma_{g',b,q}$ is $\AM_{g,1,p}$-invariant, there are induced automorphisms $g_* \in \AM_{g,1,p}$ and $\hat f_* \in \AM_{g',1,(b-1)}$ such that $g_* \mapsto \hat f_*$ by the composition \eqref{eq:suc_ex}. If $\hat f$ is isotopic to the identity relative to the $b$-th boundary component, then $\hat f_* =1 $. Hence, by Theorem~\ref{thm:inclusio} $g_* = 1$. Then $g$ is isotopic to the identity relative to the boundary. Then $f$ is isotopic to the identity relative to the boundary. 
\end{proof}

\section{Examples}
We fix $g,p$ such that $(g,p) \neq (0,2)$. Let $\hat S$ be the universal cover of $S_{g,1,p}$. The fundamental group of $S_{g,1,p}$, denoted $\Sigma_{g,1,p}$, acts on $\hat S$. Let $H$ be a subgroup of $\Sigma_{g,1,p}$ of index $m \in \naturals$. Suppose $H$ is $\AM_{g,1,p}$-invariant. The quotient space $\hat S \slash H$ is an orientable surface, denoted $S_{g',b,q}$. We identify the fundamental group of $S_{g',b,q}$, denoted $\Sigma_{g',b,q}$, with $H$. The cover $\hat S \to S_{g,1,p}$ induces a cover $S_{g',b,q} \to S_{g,1,p}$ with group of deck transformation $G \coloneq \Sigma_{g,1,p}\slash \Sigma_{g',b,q}$. If $t_k \notin \Sigma_{g',b,q}$ for all $k \in [1{\uparrow}p]$, then the corresponding cover $S_{g',b,0} \to S_{g,1,0}$ has $p$ branching points in $S_{g,1,0}$ which lift to $q$ points in $S_{g',b,0}$. By Theorem~\ref{thm:inclusio}, we have an embedding $\AM_{g,1,p} \hookrightarrow \AM_{g',1,(b-1)}$. By choosing an appropriated basis of $H$, we can compute elements in the image of this embedding from elements of $\AM_{g,1,p}$. 

The first example is classical. In the second example, we give a basis of $H$ and compute elements in the image of the embedding.\medskip

\textbf{Example 1.} Let $H$ be the kernel of the homomorphism $\Sigma_{0,1,p} \to \langle \tau \mid \tau^ 2 \rangle$ such that $t_k \mapsto \tau$ for all $k\in [1{\uparrow} p]$. It is standard to see that $H$ is a free group of rank $2p-1$ with basis $t_1^2, t_1t_2, t_1t_3, \ldots, t_1t_p, t_1\overline t_2, t_1 \overline t_3, \ldots, t_1 \overline t_p$. It is easy to see that $H$ is invariant by the generators of $\AM_{0,1,p}$ given in \eqref{eq:trena}. For $k\in [1{\uparrow}p]$, notice that $\varrho_k = t_kH$ has order $2$ in $G \coloneq \Sigma_{0,1,p} \slash H \simeq C_2$. Hence, $\langle \varrho_k\rangle$ has index $1$ in $G$ and the $k$-th puncture in $S_{g,1,p}$ lifts to one puncture in $S_{g',b,q}$. Thus, $q=p$.
\begin{enumerate}[(a).]
\item If $p$ is even, then $\Pi t_{[1{\uparrow}p]} \in H$ and $\varrho = \Pi t_{[1{\uparrow}p]} H$ has order $1$ in $G$. Hence, $\langle \varrho \rangle$ has index $2$ in $G$ and we have $b=2$. Since $\Sigma_{g',b,q}$ has rank $2g'+b-1+q$ and $H$ has rank $2p-1$, we have $2g'+2-1+p = 2p-1$ and $g' = (p-2)\slash 2$. Hence, $\AM_{0,1,p} \hookrightarrow \AM_{(p-2)\slash 2,1,1}$, if $p$ is even.
\item If $p$ is odd, then $\Pi t_{[1{\uparrow}p]} \notin H$ and $\varrho = \Pi t_{[1{\uparrow}p]} H$ has order $2$ in $G$. Hence, $\langle \varrho \rangle$ has index $1$ in $G$ and $b=1$. Since $\Sigma_{g',b,q}$ has rank $2g'+b-1+q$ and $H$ has rank $2p-1$, we have $2g'+1-1+p = 2p-1$ and $g' = (p-1)\slash 2$. Hence, $\AM_{0,1,p} \hookrightarrow \AM_{(p-1)\slash 2,1,0}$, if $p$ is odd.
\end{enumerate}

\textbf{Example 2.} Let $\Sigma_{1,1,0} = \langle x,y \mid \,\, \rangle$. Let $H$ be the kernel of the homomorphism $\Sigma_{1,1,0} \to \langle \tau_1 \mid \tau_1^2 \rangle \times \langle \tau_2 \mid \tau_2^2 \rangle$ such that $x \mapsto \tau_1,\, y \mapsto \tau_2$. It is standard to see that $H$ is a free group of rank $5$. It can be shown that $H$ is a characteristic subgroup of $\Sigma_{1,1,0}$. Notice that $\varrho = \overline x \,\,\overline y x y H$ has order $1$ in $G\coloneq \Sigma_{1,1,0} \slash H \simeq C_2 \times C_2$. Hence, $\langle \varrho \rangle$ has index  $4$ in $G$ and $b = 4$. We have $p = 0$ and $q = 0$.
Since $\Sigma_{g',b,q}$ has rank $2g'+b-1+q$ and $H$ has rank $5$, we have $2g'+4-1+0 = 5$ and $g' = 1$. Hence, $\AM_{1,1,0} \hookrightarrow \AM_{1,1,3}$. We take the following basis of $\Sigma_{1,1,3}$: $\hat x = x^2,\, \hat y = y^2,\, \hat t_1 = (\overline y\,\, \overline x yx)^{\overline x\,\, \overline y^2 x^2 y^2},\, \hat t_2 = (\overline y\,\, \overline x yx)^{y},\, \hat t_3 = (\overline y\,\, \overline x yx)^{x y}$. It is well-known that $\AM_{1,1,0} = \langle \alpha, \beta \mid \alpha \beta \alpha = \beta \alpha \beta \rangle$, where 
$$\begin{array}{ll}
\alpha \coloneq
\left\{
\begin{array}{lll}
x & \mapsto & \overline y x,\\
y & \mapsto & y,
\end{array}
\right.
& 
\beta\coloneq
\left\{
\begin{array}{lll}
x & \mapsto & x,\\
y & \mapsto & x y.
\end{array}
\right.
\end{array}
$$
A straightforward computation shows that the image of $\alpha$ and $\beta$ in $\AM_{1,1,3}$, denoted $\hat \alpha$ and $\hat \beta$, are
$$
\begin{array}{ll}
\hat \alpha\coloneq
\left\{
\begin{array}{lll}
\hat x & \mapsto & \hat y^{-1} \hat x\hat y \hat t_2 \hat t_3 \hat t_2^{-1} \hat y^{-1},\\
\hat y & \mapsto & \hat y,\\
\hat t_1 & \mapsto & \hat t_3^{\,\hat t_2^{-1} \hat y^{-1}\hat x^{-1} \hat y^{-1} \hat x \hat y \hat t_2 \hat t_3 \hat t_2^{-1}},\\
\hat t_2 & \mapsto & \hat t_2,\\
\hat t_3 & \mapsto & \hat t_1^{\,\hat t_2 \hat t_3},
\end{array}
\right.
& \beta\coloneq
\left\{
\begin{array}{lll}
\hat x & \mapsto & \hat x,\\
\hat y & \mapsto & \hat x \hat y \hat t_2,\\
\hat t_1 & \mapsto & \hat t_1^{\,\hat y^{-1} \hat x^{-1} \hat y \hat t_2^{-1} \hat y \hat x \hat y \hat t_2},\\
\hat t_2 & \mapsto & \hat t_3,\\
\hat t_3 & \mapsto & \hat t_2^{\,\hat y^{-1} \hat x \hat y \hat t_2 \hat t_3}.
\end{array}
\right.
\end{array}
$$\medskip

\textbf{Example 3.} Let $F_3 \coloneq \langle a_{[1{\uparrow}3]} \mid \, \rangle$. Let $H$ be the kernel of the homomorphism $F_3 \to \langle \tau_1 \mid \tau_1^2 \rangle \times \langle \tau_2 \mid \tau_2^2 \rangle \times \langle \tau_3 \mid \tau_3^2 \rangle$ such that $a_k \mapsto \tau_k$ for all $k\in [1{\uparrow} 3]$. It is standard to see that $H$ is a free group of rank $17$. It can be shown that $H$ is a characteristic subgroup of $F_3$. 
\begin{enumerate}[(a).]
\item We identify $\Sigma_{0,1,3}$ and $F_3$ by putting $a_k \leftrightarrow t_k$ for all $k\in [1{\uparrow}3]$.
Notice that $\varrho = t_1t_2t_3 H$ has order $2$ in $G\coloneq \Sigma_{0,1,3} \slash H \simeq C_2\times C_2 \times C_2$. Hence, $\langle \varrho\rangle$ has index $4$ in $G$ and $b=4$. On the other hand, for all $k\in [1{\uparrow}3]$, $\varrho_k = t_kH$ has order $2$ in $G$. Hence, for all $k\in [1{\uparrow}3]$, $\langle \varrho_k \rangle$ has index $4$ in $G$ and the $k$-th puncture in $S_{0,1,3}$ lifts to $4$ punctures in $S_{g',b,q}$.
Thus, $q= 12$.
Since $\Sigma_{g',b,q}$ has rank $2g'+b-1+q$ and $H$ has rank $17$, we have $2g'+4-1+12 = 17$ and $g' = 1$. Hence, $\AM_{0,1,3} \hookrightarrow \AM_{1,1,3}$.
\item We identify $\Sigma_{1,1,1}$ and $F_3$ by putting $a_1 \leftrightarrow x_1, a_2 \leftrightarrow y_1$ and $a_3 \leftrightarrow t_1$.
Notice that $\varrho = [x_1,y_1] t_1 H$ has order $2$ in $G\coloneq \Sigma_{1,1,1} \slash H\simeq C_2\times C_2 \times C_2$. Hence, $\langle \varrho\rangle$ has index $4$ in $G$ and $b=4$. On the other hand, $\varrho_1 = t_1H$ has order $2$ in $G$. Hence, $\langle \varrho_1 \rangle$ has index $4$ in $G$ and the puncture in $S_{1,1,1}$ lifts to $4$ punctures in $S_{g',b,q}$.
Thus, $q= 4$.
Since $\Sigma_{g',b,q}$ has rank $2g'+b-1+q$ and $H$ has rank $17$, we have $2g'+4-1+4 = 17$ and $g' = 5$. Hence, $\AM_{1,1,1} \hookrightarrow \AM_{5,1,3}$.
\end{enumerate}

\section{Proofs of Theorem~\ref{thm:principal} and Theorem~\ref{thm:inclusio}}

\begin{definition}\label{def:t-lliure}
An element of $\Sigma_{g,1,p}$ is said to be {\it $t$-squarefree}
if, in its reduced expression, no two consecutive terms in $t_{[1{\uparrow}p]}\vee \overline t_{[1{\uparrow}p]}$ are equal; for example: $x_1x_1t_2t_3$ is $t$\d1squarefree;  $x_1t_2t_2y_1$ is non-$t$\d1squarefree.
\end{definition}

To proof Theorem~\ref{thm:principal} we need the following theorem.

\begin{theorem}\label{thm:t-sqfree}
For every $\phi \in \AM_{g,1,p}$, the elements of $x^\phi_{[1{\uparrow}g]}\vee y^\phi_{[1{\uparrow}g]}\vee t^\phi_{[1{\uparrow}p]}$ are $t$-squarefree.
\end{theorem}

\begin{proof}{\it (of Theorem~\ref{thm:principal})}
If $p = 0$, then $\psi$ is the identity and nothing needs to be said.

Suppose $p \geq 1$. Let $a \in x_{[1{\uparrow}g]} \vee y_{[1{\uparrow}g]} \vee t_{[1{\uparrow}p]}$. If $\phi$ is an element of the kernel of $\psi:\AM_{g,1,p} \to \AM_{g,1,p^{(d)}}$, then $a^\phi$ and $a$ have the same image in $\Sigma_{g,1,p^{(d)}}$. On the other hand, by Theorem~\ref{thm:t-sqfree}, $a^\phi$ is $t$-squarefree. Hence, $a^\phi$ has the same normal form in $\Sigma_{g,1,p}$ as in $\Sigma_{g,1,p^{(d)}}$. Thus, $a^\phi = a$.
\end{proof} \medskip

Let $N_d$ be the normal closure of $t_1^d, t_2^d, \ldots, t_p^d$ in $\Sigma_{g,1,p}$. Then \[\Sigma_{g,1,p} \slash N_d = \Sigma_{g,1,p^{(d)}} = \langle x_{[1{\uparrow}g]} \vee y_{[1{\uparrow}g]} \vee \tau_{[1{\uparrow}p]} \mid \tau_{1}^d, \tau_{2}^d, \ldots, \tau_{p}^d\rangle.\]
Let $H \leq \Sigma_{g,1,p}$ be a normal subgroup of finite index such that $N_d \leq H$. Notice $H \slash N_d \leq \Sigma_{g,1,p^{(d)}}$. We set $$\AM_{g,1,p}(H) = \{\phi \in \AM_{g,1,p}\mid H^{\phi} = H\},$$ and $$\AM_{g,1,p^{(d)}}(H \slash N_d) = \{\tilde \phi \in \AM_{g,1,p^{(d)}}\mid ( H \slash N_d)^{\tilde \phi} =  H \slash N_d \}.$$

\begin{proposition}\label{prop:inclusio}
With the above notation suppose $(g,p,d) \neq (0,2,2)$. Let $\phi \in \AM_{g,1,p}(H)$. Then $\psi(\phi) \in \AM_{g,1,p^{(d)}}(H\slash N_d)$. If $\psi(\phi)|_{H\slash N_d} = 1$, then $\phi = 1$.
\end{proposition}

\begin{proof}
Since $N_d$ and $H$ are $\phi$\d1invariant, the restriction of $\psi(\phi) \in \AM_{g,1,p^{(d)}}$ to $H\slash N_d$ is an element of $\Aut(H\slash N_d)$. Hence, $\psi(\phi)\in \AM_{g,1,p^{(d)}}(H\slash N_d)$.\medskip

Since $H$ has finite index in $\Sigma_{g,1,p}$, there exists $r\in \integers,\, r\neq 0,$ such that $$(\Pi_{i\in [1{\uparrow}g]} [x_i,y_i]\Pi t_{[1{\uparrow}p]})^{r} \in H.$$ Fix $k \in [1{\uparrow}p]$. Since $H$ is normal in $\Sigma_{g,1,p}$, we see $$\overline t_k(\Pi_{i\in [1{\uparrow}g]} [x_i,y_i]\Pi t_{[1{\uparrow}p]})^{r} t_k \in H.$$ If $\psi(\phi)|_{H\slash N_d} = 1$, in $\Sigma_{g,1,p^{(d)}}$,
\begin{align*}
& \overline \tau_k(\Pi_{i\in [1{\uparrow}g]} [x_i,y_i]\Pi \tau_{[1{\uparrow}p]})^{r} \tau_k\\
= & (\overline \tau_k(\Pi_{i\in [1{\uparrow}g]} [x_i,y_i]\Pi \tau_{[1{\uparrow}p]})^{r} \tau_k)^{\psi(\phi)}\\
= & \overline \tau_k^{\,\psi(\phi)} (\Pi_{i\in [1{\uparrow}g]} [x_i,y_i]\Pi \tau_{[1{\uparrow}p]})^{r} \tau_k^{\psi(\phi)}.
\end{align*}

Hence, in $\Sigma_{g,1,p^{(d)}}$, $\tau_k^{\psi(\phi)} \overline \tau_k$ commutes with $(\Pi_{i\in [1{\uparrow}g]} [x_i,y_i]\Pi \tau_{[1{\uparrow}p]})^{r}$. Then $\tau_k^{\psi(\phi)} \overline \tau_k \in \langle \Pi_{i\in [1{\uparrow}g]} [x_i,y_i]\Pi \tau_{[1{\uparrow}p]} \rangle$, and,
\begin{equation}\label{eq:01}
\tau_k^{\psi(\phi)} = (\Pi_{i\in [1{\uparrow}g]} [x_i,y_i]\Pi \tau_{[1{\uparrow}p]})^{r'}\tau_k,
\end{equation}
for some $r'\in \integers$. Recall $[\tau_k^{\psi(\phi)}] = [\tau_{k'}]$, for some $k'\in [1{\uparrow}p]$. If $(g,p) \neq (0,1)$, and if $(g,p,d) \neq (0,2,2)$, then \eqref{eq:01} implies $r' = 0$ and $\tau_k^{\psi(\phi)} = \tau_k$. \medskip

Recall $\Sigma_{g,1,p^{(d)}} = \Sigma_{g,1,0} \ast C_d^{\ast p}$. Fix $a\in x_{[1{\uparrow}g]} \vee y_{[1{\uparrow}g]}$. Since $H$ has finite index in $\Sigma_{g,1,p}$, there exists $s\in \integers$ such that $a^{s} \in H$. If $\psi(\phi)|_{H \slash N_d} = 1$, then $(a^s)^{\psi(\phi)} = a^{s}$, and, $a^{\psi(\phi)} = a$.\medskip

Since $\Sigma_{g,1,p^{(d)}} = \Sigma_{g,1,0} \ast C_d^{\ast p}$, $a^{\psi(\phi)} = a$ for all $a\in x_{[1{\uparrow}g]} \vee y_{[1{\uparrow}g]}$, and, $\tau_k^{\psi(\phi)} = \tau_k$ for all $k\in [1{\uparrow}p]$, we see $\psi(\phi) = 1$. By Theorem~\ref{thm:principal}, $\phi = 1$.
\end{proof}\medskip

\begin{lemma}\label{lema:normal}
With the notation above Theorem~\ref{thm:inclusio}, if $\Sigma_{g',b,q}$ is $\AM_{g,1,p}$-invariant then the normal closure of $t_1^{d}, t_2^{d}, \ldots, t_p^{d}$ in $\Sigma_{g,1,p}$ equals the normal closure of $\hat t_1, \hat t_2, \cdots ,\hat t_q$ in $\Sigma_{g',b,q}$.
\end{lemma}

\begin{proof} Recall \begin{equation}\label{eq:normal}
\{\hat t_1, \hat t_2, \ldots, \hat t_q\} = \bigvee_{k\in [1{\uparrow}p]} \{(t_k^{d})^{u_{1,k}}, (t_k^{d})^{u_{2,k}}, \ldots, (t_k^{d})^{u_{m\slash d,k}}\}.
\end{equation}
By \eqref{eq:normal}, the normal closure of $\hat t_1, \hat t_2, \cdots ,\hat t_q$ in $\Sigma_{g',b,q}$ is a subgroup of the normal closure of $t_1^{d}, t_2^{d}, \ldots t_p^{d}$ in $\Sigma_{g,1,p}$.

Let $k \in [1{\uparrow}p]$ and $w \in \Sigma_{g,1,p}$. By \eqref{eq:normal}, it is enough to proof $(t_k^{d})^w = (t_k^{d})^{u_{j,k} v}$ for some $j \in [1{\uparrow}(m \slash d)]$ and $v\in \Sigma_{g',b,q}$. Recall $G = \Sigma_{g,1,p} \slash \Sigma_{g',b,q}$, $\varrho_k = t_k \Sigma_{g',b,q} \in G$ and $G = \langle \varrho_k\rangle\rho_{1,k} \vee \langle \varrho_k\rangle\rho_{2,k} \cdots \vee \langle \varrho_k\rangle\rho_{m \slash d,k}$, where $\rho_{j,k} = u_{j,k} \Sigma_{g',b,q} \in G$ for all $j \in [1{\uparrow}(m \slash d)]$. Let $j \in [1{\uparrow}(m \slash d)]$ such that $w \Sigma_{g',b,q} \in \langle\varrho_k\rangle\rho_{j,k}$. Let $r\in [1{\uparrow} d]$ such that $w \Sigma_{g',b,q} = \varrho_k^r \rho_{j,k} = t_k^r u_{j,k}\Sigma_{g',b,q}$. Then $w = t_k^r u_{j,k}v$, for some $v \in \Sigma_{g',b,q}$ and $(t_k^{d_k})^w = (t_k^{d_k})^{t_k^r u_{j,k} v} = (t_k^{d_k})^{u_{j,k} v}$.
\end{proof}

\begin{proof}{\it (of Theorem~\ref{thm:inclusio})}
Recall $N_d$ is the normal closure in $\Sigma_{g,1,p}$ of $t_1^{d}, t_2^{d}, \ldots, t_p^{d}$. By Lemma~\ref{lema:normal}, $N_d$ is the normal closure in $\Sigma_{g',b,q}$ of $\hat t_1, \hat t_2, \cdots, \hat t_q$. Hence, $\Sigma_{g',b,0} = \Sigma_{g',b,q} \slash N_d$. We identify $\Sigma_{g',b,0}$ with $\Sigma_{g',1,b-1}$ by identifying $z_k$ with $t_k$ for all $k\in [1{\uparrow}(b-1)]$. Hence, $\Sigma_{g',1,(b-1)} = \Sigma_{g',b,q} \slash N_d$. Since $\Sigma_{g,1,p^{(d)}} = \Sigma_{g,1,p} \slash N_d$, the natural homomorphism $\Sigma_{g,1,p} \to \Sigma_{g,1,p^{(d)}}$ restricts to the natural homomorphism $\Sigma_{g',b,q} \to \Sigma_{g',1,(b-1)}$.

Let $\phi \in \AM_{g,1,p}$. Since $\psi: \AM_{g,1,p} \to \AM_{g,1,p^{(d)}}$ is given by the natural homomorphism $\Sigma_{g,1,p} \to \Sigma_{g,1,p^{(d)}}$, we see $\psi(\phi):\Sigma_{g,1,p^{(d)}} \to \Sigma_{g,1,p^{(d)}}$ completes the following commutative square 
$$
\begin{array}{ccc}
\Sigma_{g,1,p} & \stackrel{\phi}{\longrightarrow} & \Sigma_{g,1,p}\\
\downarrow & & \downarrow\\
\Sigma_{g,1,p^{(d)}} & \stackrel{\psi(\phi)}{\longrightarrow} & \Sigma_{g,1,p^{(d)}}
\end{array}
$$
where the vertical arrows are the natural homomorphisms. Since $\Sigma_{g',1,(b-1)} = \Sigma_{g',b,q} \slash N_d$ and $\Sigma_{g',b,q}$ is $\AM_{g,1,p}$-invariant, by Proposition~\ref{prop:inclusio}, there exists the restriction $\psi(\phi)|_{\Sigma_{g',1,(b-1)}}: \Sigma_{g',1,(b-1)} \to \Sigma_{g',1,(b-1)}$. Notice $\psi(\phi)|_{\Sigma_{g',1,(b-1)}}: \Sigma_{g',1,(b-1)} \to \Sigma_{g',1,(b-1)}$ completes the following commutative square 
$$
\begin{array}{ccc}
\Sigma_{g',b,q} & \stackrel{\phi|_{\Sigma_{g',b,q}}}{\longrightarrow} & \Sigma_{g',b,q}\\
\downarrow & & \downarrow\\
\Sigma_{g',1,(b-1)} & \stackrel{\psi(\phi)|_{\Sigma_{g',1,(b-1)}}}{\longrightarrow} & \Sigma_{g',1,(b-1)}
\end{array}
$$
where the vertical arrows are the natural homomorphisms. Since the second homomorphism of \eqref{eq:suc_ex} is given by eliminating the last $q$ punctures; that is, by the natural homomorphism $\Sigma_{g',b,q} \to \Sigma_{g',1,(b-1)}$, in the composition of \eqref{eq:suc_ex} we have $$\phi \mapsto \phi|_{\Sigma_{g',b,q}} \mapsto \psi(\phi)|_{\Sigma_{g',1,(b-1)}}.$$
By Proposition~\ref{prop:inclusio}, if $\psi(\phi)|_{\Sigma_{g',1,(b-1)}} = 1$, then $\phi = 1$.
\end{proof}

The rest of the paper is dedicated to proof Theorem~\ref{thm:t-sqfree}. The proof is similar to \cite[7.6 Corollary]{BD}. Notice Theorem~\ref{thm:t-sqfree} is trivial if $p = 0$.

\section{McCool's Groupoid}
\label{sec:mccool}
For the rest of the paper we suppose $p \geq 1$.\medskip

Let $n\coloneq 2g+p$, and, let $F_n$ be the free group on $X$, where $X$ is a set with $n$ elements.

\begin{notation}
Let $w\in F_n$. In this section we will denote by $[w]$ the cyclic word of $w$.
\end{notation}

\begin{definitions}
Let $T$ be a set of words and cyclic words of $F_n$. Suppose the elements of $T$ are reduced and cyclically reduced, respectively. We define the {\it Whitehead graph of $T$} as the graph with vertex set $X\vee \overline{X}$, and, one edge from $a \in X\vee \overline{X}$ to $b \in X\vee \overline{X}$ for every subword $\overline{a}b$ which appears in $w$ or $[u]$, where $w$ and $[u]$ are elemets of $T$. We say that $a$ is the initial vertex and $b$ is the terminal vertex of the edges corresponding to the subword $\overline a b$. Repetitions have to be considered. For example, since the subword $\overline a b$ appears twice in $\overline a b \overline a b$, the Whitehead graph of $\{\overline a b \overline a b\}$ has 2 edges from $a$ to $b$ (and one edge from $\overline b$ to $\overline a$). Notice that the cyclic word $[a]$ produces an edge from $\overline a$ to $a$ in the Whitehead graph.\medskip

We say that $T$ is a {\it surface word set} if the Whitehead graph of $T$ is an oriented segment, that is, the Whitehead graph of $T$ is connected with exactly $2n-1$ edges, every vertex but one is the {\it initial vertex} of exactly one edge, and, every vertex but one is the {\it terminal vertex} of exactly one edge.
\end{definitions}

\begin{example}\label{ex:whitehead}
Let $F_4 \coloneq \langle a,b,c,d \mid \, \rangle$.
\begin{enumerate}[(i).]
\item\label{ex:whitehead_1} Let $T \coloneq \{\overline a dc\overline b, [\,\overline d b], [\overline c a]\}$. The Whitehead graph of $T$ is $$\overline a \to \overline c \to \overline b \to \overline d \to c \to a \to d \to b.$$ Hence, $T$ is a surface word set.
\item Let $T \coloneq \{\overline a dc\overline b, \overline d b, [\overline ca]\}$. The Whitehead graph of $T$ is $$\overline a \to \overline c \to \overline b \phantom{\to} \overline d \to c \to a \to d \to b.$$ Hence, $T$ is not a surface word set.
\item Let $T \coloneq \{\overline a dc\overline b, d c, [\,\overline d b], [\overline c a]\}$. The Whitehead graph of $T$ is $$\overline a \to \overline c \to \overline b \to \overline d \rightrightarrows c \to a \to d \to b.$$ Hence, $T$ is not a surface word set.
\end{enumerate}
\end{example}

We illustrate the following remarks with examples in $F_4 = \langle a, b, c, d \mid \,\rangle$.
\begin{remarks} Let $T$ be a surface word set.
\begin{enumerate}
\item The Whitehead graph of $T$ defines a sequence $a_{[1{\uparrow}2n]}$ with underling set $X\vee \overline X$ such that for all $i\in [1{\uparrow}(2n-1)]$, the Whitehead graph of $T$ has exactly one edge with initial vertex $a_i$ and terminal vertex $a_{i+1}$, equivalently, $\overline a_i a_{i+1}$ is a subword of exactly one element of $T$. We say that $a_{[1{\uparrow}2n]}$ is the {\it associated sequence} of $T$. 

In Example~\ref{ex:whitehead}\eqref{ex:whitehead_1}, the associated sequence of $T$ is $(\overline a, \overline c, \overline b, \overline d, c, a, d, b)$.

\item We can recover $T$ from the associated sequence of $T$. The process to recover $T$ from its associated sequence is the invers process to construct the Whitehead graph. We give two examples below. From this process, it is easy to see that $T$ has exactly one word, and, all other elements of $T$ are cyclic words.

In $F_4$, from the sequence $(a,b,c,d,\overline a, \overline b, \overline c, \overline d)$ we have the surface word set $\{a \overline b c \overline d \,\overline a b \overline c d\}$, and, from the sequence $(a, b, c, d, \overline d, \overline c, \overline b, \overline a)$ we have the surface word set $\{a, [b\overline a], [c \overline b\,], [d \overline c], [\,\overline d\,]\}$.

\item Let $p$ be the cardinality of $T$ minus one. We say that $T$ is a $(g,p)$-surface word set, where $g=(n-p)\slash 2$.  By induction on $n$, it can be seen that $n\geq p$ and $n-p$ is even. Hence, $g\in \naturals$.
\end{enumerate}
\end{remarks}

\begin{definition}
Let $\phi \in \Aut(F_n)$.

We say that $\phi$ is a {\it type-$1$ Nielsen automorphism} if $\phi$ restricts to a permutation of $X\vee \overline X$.

We say that $\phi$ is a {\it type-$2$ Nielsen automorphism} if there exists $a,b\in X\vee \overline X$ such that $a\neq b, \overline b$ and \medskip

$
\phi:=\left\{\begin{array}{lll}
a & \mapsto & ab,\\
c & \mapsto & c \quad \textrm{ for all } c\in X,\, c\neq a^{\pm 1}.
\end{array}\right.$
\newline
We denote $\phi$ by $(a \mapsto ab)$ or $(\overline a \mapsto \overline b \overline a)$.
\end{definition}

\begin{definition}
Let $\G_{g,p}$ be the groupoid with objects $(g,p)$-surface word sets, and, given $T_1,\, T_2$ two $(g,p)$-surface word sets $$\Hom(T_1,T_2)\coloneq \{\phi\in \Aut(F_n)\mid T_1^\phi = T_2\},$$
where $T_1^\phi \coloneq \{w^\phi,[u^{\phi}]\mid w,[u]\in T_1\}$. Here, $w^\phi$ is reduced and $[u^\phi]$ is cyclically reduced. Hence, $[v] = [u^\phi]$ means that $v$ and $u^\phi$ are conjugated.\medskip

We say that $(T_1,T_2,\phi)\in \Hom(T_1, T_2)$ is a {\it type-$1$ Nielsen} of $\G_{g,p}$ if $\phi$ is a type-$1$ Nielsen automorphism. Similarly, for type-$2$ Nielsen automorphisms. We say that $(T_1,T_2,\phi)\in \Hom(T_1, T_2)$ is a Nielsen if it is either a type-$1$ Nielsen or a type-$2$ Nielsen.
\end{definition}

We illustrate the following remarks with examples in $F_4 = \langle a, b, c, d \mid \, \rangle$.
\begin{remark}
Let $(T_1,T_2,\phi)$ be a Nielsen of $\G_{g,p}$.
\begin{enumerate}\label{rem:N-seq}
\item\label{rem:N-seq_1} If $(T_1,T_2,\phi)$ is a type-$1$ Nielsen, then the associated sequence of $T_2$ is obtained from the associated sequence of $T_1$ by applying the permutation $\phi$ to every element of the sequence.

In $F_4$, let $T_1 = \{a\overline d\,\, \overline b c, [\,\overline a b], [\, \overline c d]\}$. Notice the associated sequence of $T_1$ is $(a, b, c, d, \overline b, \overline a, \overline d, \overline c)$. If $\phi\coloneq (a \mapsto \overline b, b \mapsto c, c \mapsto \overline a, d \mapsto \overline d)$, then the associated sequence of $T_2$ is $(\overline b, c, \overline a, \overline d, \overline c, b, d, a)$.
\item\label{rem:N-seq_2} Suppose $(T_1,T_2,\phi)$ is a type-$2$ Nielsen. Then $\phi = (a_i \mapsto b a_i)$ for some $i \in [1{\uparrow}2n]$, $b\in X\vee \overline X$ such that $a_i\neq  b,\overline b$. Since in the Whitehead graph of $T$ there are exactly $2n-1$ edges, there exists $w\in T_1$ or $[u]\in T_1$ such that applying $\phi$ to $w$ or $[u]$ produces a cancellation. If the cancellation appears from he subword $\overline a_{i-1}a_i$, then $b=a_{i-1}$. If the cancellation appears from the subword $\overline a_i a_{i+1}$, then $b = a_{i+1}$. Hence, either $\phi = (a_i\mapsto a_{i-1}a_i)$ for some $i \in [2{\uparrow}2n],\, a_i\neq \overline a_{i-1}$; or $\phi = (\overline a_i\mapsto \overline a_i \overline a_{i+1})$ for some $i \in [1{\uparrow}(2n-1)],\, a_i\neq \overline a_{i+1}$. In the former case the associated sequence of $T_2$ is obtained from the associated sequence of $T_1$ by moving $a_i$ from immediately after $a_{i-1}$ to immediately before $\overline a_{i-1}$. In the later case  the associated sequence of $T_2$ is obtained from the associated sequence of $T_1$ by moving $a_i$ from immediately before $a_{i+1}$ to immediately after $\overline a_{i+1}$.

In $F_4$, let $T_1 = \{a\overline b c \overline d \,\overline a b \overline c d\}$. Notice the associated sequence of $T_1$ is $(a, b, c, d, \overline a, \overline b, \overline c, \overline d)$. If $\phi\coloneq (b \mapsto ab)$, then the associated sequence of $T_2$ is $(a, c, d, b, \overline a, \overline b, \overline c, \overline d)$. In fact $(a\overline b c \overline d \,\overline a b \overline c d)^{(b \mapsto ab)} = a\overline b \,\overline a c \overline d  b \overline c d$. If $\phi\coloneq (\overline a \mapsto \overline a \overline b)$, then the associated sequence of $T_2$ is $(b, c, d, \overline a, \overline b, a, \overline c, \overline d)$. In fact $(a\overline b c \overline d \,\overline a b \overline c d)^{(\overline a \mapsto \overline a \overline b)} = ba\overline b c \overline d \,\overline a \,\,\overline c d$.
\end{enumerate}
\end{remark}

\begin{remark}
It is easy to see $\{\Pi_{i\in [1{\uparrow}g]}[x_i,y_i]\Pi_{j\in [1{\uparrow}p]}t_j,[\,\overline t_1], [\,\overline t_2],\ldots, [\,\overline t_p]\}$ is a $(g,p)$-surface word set of $\Sigma_{g,1,p}$. Its associated sequence is $$(\overline x_1, y_1, x_1, \overline y_1, \overline x_2, y_2, x_2, \overline y_2, \ldots, \overline x_g, y_g, x_g, \overline y_g, t_1, \overline t_1, t_2, \overline t_2, \ldots, t_p, \overline t_p).$$
We say that $\{\Pi_{i\in [1{\uparrow}g]}[x_i,y_i]\Pi_{j\in [1{\uparrow}p]}t_j,[\,\overline t_1], [\,\overline t_2],\ldots, [\,\overline t_p]\}$ is the standard $(g,p)$- surface word set of $\Sigma_{g,1,p}$.
\end{remark}

\begin{remark}\label{rem:hom_out} $\AM_{g,1,p} = \Hom(T,T)$, where $T$ is the standard $(g,p)$-surface word set of $\Sigma_{g,1,p}$.
\end{remark}

\begin{theorem}[McCool~\cite{McCool},\cite{DF97}]\label{t:mccool} $\G_{g,p}$ is generated by Nielsen elements.
\end{theorem}

\section{Ends of free group}
\label{sec:ends}
Let $n\coloneq 2g+p$ and $F_n$ is the free group on $X$, where $\abs{X} = n$.

\begin{notation}\label{not:basic} Let $a_{[1{\uparrow}k]}$
be the normal form for $w\in F_n$; in particular, $w= \Pi a_{[1{\uparrow}k]}$.
The set of elements of $F_n$ whose normal forms have $a_{[1{\uparrow}k]}$
 as an initial segment is denoted $(w{\star})$; and, the set of
elements of $F_n$ whose normal forms have $a_{[1{\uparrow}k]}$
 as a terminal segment is denoted  $({\star} w)$.  The elements of $(w{\star})$ are said to {\it begin with } $w$,
and the elements of $({\star} w)$  are said to {\it end with}~$w$.
\end{notation}

\begin{review}  An {\it end} of $F_n$ is a sequence
$a_{[1{\uparrow}\infty{[}}$ in $X \vee \overline X$
such that, for each $i \in \left[1{\uparrow}\infty\right[ $\,\,, $a_{i+1} \ne \overline a_i$.
We represent $a_{[1{\uparrow}\infty{[}}$ as a formal right-infinite reduced product,
$a_1a_2\cdots$ or $\Pi a_{[1{\uparrow}\infty{[}}$.

We denote the set of ends of $F_n$ by $\partial F_n$.

For each $w \in F_n$,
we define the {\it shadow} of $w$ in $\partial F_n$ to be
$$(w{\blacktriangleleft}) \coloneq \{a_{[1{\uparrow}\infty[} \in \partial F_n \mid
\text{ $\Pi a_{[1{\uparrow}\abs{w}\,]} = w$} \}.$$
Thus, for example, $(1{\blacktriangleleft}) = \partial F_n$.\medskip
\end{review}

\begin{definition}\label{def:ends}
Let $T$ be a surface word set. We now give  $\partial F_n$ an ordering, $<_T$, with respect to $T$ as follows. Let $a_{[1{\uparrow}2n]}$ be the associated sequence of $T$. Recall $a_{[1{\uparrow}2n]}$ is a listing of the elements of $X\vee \overline X$ . 
For each $w \in F_n$, we assign
an ordering, $<_T$, to a partition of $(w{\blacktriangleleft})$ into $2n$ or $2n-1$ subsets, depending
as $w=1$ or $w\ne 1$,  as follows.  We set
\begin{align*}
&(a_1{\blacktriangleleft}) <_T (a_2{\blacktriangleleft}) <_T
(a_3{\blacktriangleleft}) <_T \cdots
<_T (a_{2n-1}{\blacktriangleleft}) <_T (a_{2n}{\blacktriangleleft}).
\intertext{If $i \in [1{\uparrow}n]$ and
$w \in  ({\star}\overline a_i)$, then we set}
&(w a_{i+1}{\blacktriangleleft}) <_T  (w a_{i+2}{\blacktriangleleft})
 <_T (w a_{i+3}{\blacktriangleleft}) <_T \cdots\\
&\quad \cdots <_T(w a_{2n-1}{\blacktriangleleft}) <_T (w a_{2n}{\blacktriangleleft}) <_T
(w a_{1}{\blacktriangleleft}) <_T (w a_{2}{\blacktriangleleft}) <_T
(w a_3{\blacktriangleleft}) <_T \cdots \\
&\quad \cdots <_T (w a_{i-2}{\blacktriangleleft}) <_T (w a_{i-1}{\blacktriangleleft}) .
\end{align*}
Hence, for each $w \in F_n$, we have
an ordering $<_T$ of a partition of $(w{\blacktriangleleft})$ into $2n$ or $2n-1$ subsets.

If $ b_{[1{\uparrow}\infty[}$ and $ c_{[1{\uparrow}\infty[}$
are two different ends, then there exists $i \in \naturals$ such that
$b_{[1{\uparrow}i]} = c_{[1{\uparrow}i]}$ and
$b_{i+1} \ne c_{i+1}$.
Let $w = \Pi b_{[1{\uparrow}i]} = \Pi c_{[1{\uparrow}i]}$ in $F_n$. Then
$b_{[1{\uparrow}\infty[}$ and $c_{[1{\uparrow}\infty[}$ lie in $(w{\blacktriangleleft})$,
but lie in different elements of the partition of
$(w{\blacktriangleleft})$ into $2n$ or $2n-1$ subsets.
We then order $b_{[1{\uparrow}\infty[}$ and $c_{[1{\uparrow}\infty[}$ using the
order of the elements of the partition of
$(w{\blacktriangleleft})$ that they belong to.
This completes the definition of the ordering $<_T$ of $\partial F_n$.\medskip

Let $w$ be the non-cyclic element of $T$. We remark that in $(\partial F_n, <_T)$ the smallest element is $w^\infty$ and the largest element is
$\overline w^\infty$.

For example, in $F_4 = \langle a,b,c,d\mid \,\rangle$ we take $T = \{a\overline d\,\, \overline bc,[\overline a b],[\overline c d]\}$. The associated sequence of $T$ is $(a,b,c,d,\overline b, \overline a, \overline d, \overline c)$. In $(\partial F_4, <_{T})$, the smallest element is $(a\overline d\,\,\overline b c)^\infty$, and, the largest element is $(\overline c b d \overline a)^{\infty}$.
\end{definition}

\begin{notation}
We denote by $<$ the order on $\partial\Sigma_{g,1,p}$ with respect to the standard $(g,p)$-surface word set of $\Sigma_{g,1,p}$. 
\end{notation}

\begin{review} 
Let $\hat S$ be the universal cover of $S_{g,1,p}$. Suppose $S_{g,1,p}$ has negative Euler characteristic, that is, $2g+p\geq 2$. Then $\hat S$ can be identified with the hyperbolic plane. Let $\partial \hat S$ be the boundary of $\hat S$. It is well-known that $\partial \hat S$ can be identified with $\reals \vee \{\infty\}$. Let $\ast$ be the point in $\partial \hat S$ corresponding to $\infty$ by this identification.
The identification between $\partial \hat S$ and $\reals \vee \{\infty\}$ restricts to an identification between $\partial \hat S - \{\ast\}$ and $\reals$.
By work of Nielsen-Thurston~\cite{Cooper},~\cite{ShortWiest},
there is an action of $\M_{g,1,p}$ on $\partial \hat S$ with a fixed point, which we can suppose to be $\ast \in \partial \hat S$. Hence, an action of $\M_{g,1,p}$ on $\reals$. By~\cite{ShortWiest}, this action preserves the usual order of $\reals$. Remark~\ref{rem:hom_out} and Proposition~\ref{prop:order}
give the analog statement for $\AM_{g,1,p}$ and $\partial \Sigma_{g,1,p}$.
\medskip

Let $\phi \in \Aut(F_n)$. It is proved in~\cite{Cooper} that $(\Pi a_{[1{\uparrow}\infty[})^\phi = \lim_{k\to \infty} (\Pi a_{[1{\uparrow}k]})^\phi$ defines a map $\partial F_n \to \partial F_n$, which we still denote by $\phi$.

\begin{proposition}\label{prop:order}
Let $T_1,T_2$ be surface word sets of $F_n$ and $(T_1,T_2,\phi) \in \Hom(T_1,T_2)$. Then $\phi: (\partial F_n, \le_{T_1}) \to (\partial F_n, \le_{T_2})$ respects the orderings.
\end{proposition}

\begin{proof}
By Theorem~\ref{t:mccool}, we can restrict ourselves to the case where $(T_1, T_2, \phi)$ is a Nielsen.

By Remark~\ref{rem:N-seq}\ref{rem:N-seq_1}, the result is clear if $(T_1, T_2, \phi)$ is a type-$1$ Nielsen. Hence, we suppose $(T_1, T_2, \phi)$ is a type-$2$ Nielsen.

Let $a_{[1{\uparrow}2n]}$ be the associated sequence of $T_1$. Then either $\phi=(a_i\mapsto a_{i-1}a_i)$ for some $i\in [2{\uparrow}2n],\, a_i \neq \overline a_{i-1}$, or, $\phi=(\overline a_i\mapsto \overline a_{i}\overline a_{i+1})$ for some $i\in [1{\uparrow}(2n-1)]$, $a_i \neq \overline a_{i+1}$.

Suppose $\phi=(a_i\mapsto a_{i-1}a_i)$ for some $i\in [2{\uparrow}2n],\, a_i \neq \overline a_{i-1}$.\medskip

The following correspondence by the action of $(a_i \mapsto a_{i-1}a_i)$ is clear.\medskip

\centerline{
\begin{tabular}
{
>{$}l<{$}
@{\hskip.5cm}
>{$}l<{$}
@{\hskip.5cm}
>{$}c<{$}
@{\hskip.5cm}
>{$}l<{$}
}
\setlength\extrarowheight{60pt}
& & (a_i \mapsto a_{i-1}a_i) &
\\[.08cm]& (\star\, \overline a_i a_{i-1}) & \longrightarrow & (\star\, \overline a_{i}),
\\[.08cm]& [(\star\, a_{i-1})-(\star\, \overline a_i a_{i-1})] & \longrightarrow & (\star\, a_{i-1}),
\\[.08cm] & (\star\, a_i) & \longrightarrow & (\star\, a_i),
\\[.08cm] & (\star\, a_{k}) & \longrightarrow & (\star\, a_{k}), \quad a_{k}\neq a_{i-1}^{\pm 1}, a_i^{\pm 1},
\\[.1cm] & (\star\, \overline a_{i-1}) & \longrightarrow & (\star\, \overline a_{i-1}),
\\[.08cm]& (\star\, \overline a_i) & \longrightarrow & (\star\,\overline a_i\overline a_{i-1}).
\end{tabular}
}

\bigskip

The following correspondence by the action of $(a_i \mapsto a_{i-1}a_i)$ is clear.\medskip

\centerline{
\begin{tabular}
{
>{$}l<{$}
@{\hskip.5cm}
>{$}l<{$}
@{\hskip.5cm}
>{$}c<{$}
@{\hskip.5cm}
>{$}l<{$}
}
\setlength\extrarowheight{60pt}
&  & (a_i \mapsto a_{i-1}a_i) & 
\\[.08cm]& (a_{i-1}\blacktriangleleft) & \longrightarrow & (a_{i-1}\blacktriangleleft) - (a_{i-1}a_i\blacktriangleleft),
\\[.08cm] & (a_i \blacktriangleleft) & \longrightarrow & (a_{i-1}a_i \blacktriangleleft),
\\[.08cm] & (a_{k}\blacktriangleleft) & \longrightarrow & (a_{k}\blacktriangleleft), \quad a_{k}\neq a_{i-1}^{\pm 1}, a_i^{\pm 1},
\\[.08cm] & (\overline a_{i-1} a_i\blacktriangleleft) & \longrightarrow & (a_i\blacktriangleleft),
\\[.08cm] & [(\overline a_{i-1}\blacktriangleleft)-(\overline a_{i-1} a_i\blacktriangleleft)] & \longrightarrow & (\overline a_{i-1}\blacktriangleleft),
\\[.08cm]& (\overline a_i \blacktriangleleft) & \longrightarrow & (\overline a_i\blacktriangleleft).
\end{tabular}}

\bigskip

From the first row of the first table and the second table we deduce the following table.
\medskip

\centerline{
\begin{tabular}
{
>{$}l<{$}
@{\hskip.5cm}
>{$}l<{$}
@{\hskip.5cm}
>{$}c<{$}
@{\hskip.5cm}
>{$}l<{$}
}
\setlength\extrarowheight{60pt}
&& (a_{i} \mapsto a_{i-1}a_i) &
\\[.08cm]& (\star\,  \overline a_i a_{i-1})(a_{i-1}\blacktriangleleft)  & \longrightarrow & (\star\, \overline a_{i} )[(a_{i-1}\blacktriangleleft)-(a_{i-1}a_i\blacktriangleleft)],
\\[.08cm]
& (\star\, \overline a_i a_{i-1})(a_{i}\blacktriangleleft)  & \longrightarrow & (\star\, \overline a_{i} )(a_{i-1}a_i\blacktriangleleft),
\\[.08cm]
& (\star\, \overline a_i a_{i-1})(a_{k}\blacktriangleleft)  & \longrightarrow & (\star\, \overline a_{i} )(a_{k}\blacktriangleleft), \quad a_{k} \neq a_{i-1}^{\pm 1}, a_i^{\pm 1},
\\[.08cm]
& (\star\, \overline a_i a_{i-1})(\overline a_{i}\blacktriangleleft)  & \longrightarrow & (\star\, \overline a_{i} )(\overline a_i\blacktriangleleft).
\end{tabular}}

\bigskip
Notice the cases $(\star \overline a_{i}a_{i-1})(\overline a_{i-1} a_i \blacktriangleleft)$ and $(\star \overline a_{i}a_{i-1})[(\overline a_{i-1} \blacktriangleleft)-(\overline a_{i-1} a_i \blacktriangleleft)]$ do not have to be considered since they are not in reduced form.\medskip

Let $\mathfrak{e},\mathfrak{f}\in \partial F_n$ such that $\mathfrak{e} = (w \overline a_{i}a_{i-1}) \mathfrak{e}',\, \mathfrak{f} = (w \overline a_{i}a_{i-1}) \mathfrak{f}'$ and the first letter of $\mathfrak{e}'$ is different from the first letter of $\mathfrak{f}'$. Let $j\in [1{\uparrow}2n]$ such that $a_j = \overline a_{i-1}$.  
By the third table, $\mathfrak{e}^{(a_i\mapsto a_{i-1}a_i)} = (u \overline a_{i}) \mathfrak{e}'',\, \mathfrak{f}^{(a_i\mapsto a_{i-1}a_i)} = (u \overline a_{i}) \mathfrak{f}''$ in reduced form. Let $b_{[1{\uparrow}2n]}$ be the associated sequence of $T_2$. Recall $b_{[1{\uparrow}2n]}$ is obtained from $a_{[1{\uparrow}2n]}$ by moving $a_i$ from immediately after $a_{i-1}$ to immediately before $a_j = \overline a_{i-1}$. There are two cases according to $j < i-1$ or $i-1 < j$.

If $j < i-1$, then 
$$\begin{array}{ll}
b_{[1{\uparrow}(j-1)]} & = a_{[1{\uparrow}(j-1)]},\\
b_j & = a_i,\\
b_{[(j+1){\uparrow}i]} & = a_{[j{\uparrow}(i-1)]},\\
b_{[(i+1){\uparrow}2n]} & = a_{[(i+1){\uparrow}2n]}.
\end{array}$$
The partition with respect to $a_{[1{\uparrow}2n]}$ of $(\overline a_{j}\blacktriangleleft) = (a_{i-1}\blacktriangleleft)$ is $(a_{j+1}\blacktriangleleft),$ \linebreak $(a_{j+2}\blacktriangleleft),\ldots, (a_{i-1}\blacktriangleleft)$,$(a_{i}\blacktriangleleft),(a_{i+1}\blacktriangleleft), \ldots, (a_{2n}\blacktriangleleft),(a_{1}\blacktriangleleft),(a_{2}\blacktriangleleft), \ldots ,(a_{j-1}\blacktriangleleft)$. The partition with respect to $b_{[1{\uparrow}2n]}$ of $(\overline a_{i}\blacktriangleleft)$ is $(a_{j}\blacktriangleleft),(a_{j+1}\blacktriangleleft), \ldots ,(a_{i-1}\blacktriangleleft)$, $(a_{i+1}\blacktriangleleft),(a_{i+2}\blacktriangleleft), \ldots ,(a_{2n}\blacktriangleleft),(a_{1}\blacktriangleleft),(a_{2}\blacktriangleleft), \ldots ,(a_{j-1}\blacktriangleleft)$. By the third table,
\medskip

\centerline{
\begin{tabular}
{
>{$}l<{$}
@{\hskip.5cm}
>{$}l<{$}
@{\hskip.5cm}
>{$}c<{$}
@{\hskip.5cm}
>{$}l<{$}
}
\setlength\extrarowheight{60pt}
&  & (a_i \mapsto a_{i-1}a_i) & 
\\[.08cm]& (w \overline a_i a_{i-1})(a_{j+1}\blacktriangleleft) & \longrightarrow & (u \overline a_i)(a_{j+1}\blacktriangleleft),
\\[.08cm]& (w \overline a_i a_{i-1})(a_{j+2}\blacktriangleleft) & \longrightarrow & (u \overline a_i)(a_{j+2}\blacktriangleleft),
\\[.03cm]&& \vdots
\\[.08cm]& (w \overline a_i a_{i-1})(a_{i-2}\blacktriangleleft) & \longrightarrow & (u \overline a_i)(a_{i-2}\blacktriangleleft),
\\[.08cm] & (w \overline a_i a_{i-1})(a_{i-1} \blacktriangleleft) & \longrightarrow & (u \overline a_i)[(a_{i-1} \blacktriangleleft)-(a_{i-1}a_i \blacktriangleleft)],
\\[.08cm] & (w \overline a_i a_{i-1})(a_i \blacktriangleleft) & \longrightarrow & (u \overline a_i)(a_{i-1}a_i \blacktriangleleft),
\\[.08cm] & (w \overline a_i a_{i-1})(a_{i+1}\blacktriangleleft) & \longrightarrow & (u \overline a_i)(a_{i+1}\blacktriangleleft),
\\[.03cm]&& \vdots
\\[.08cm] & (w \overline a_i a_{i-1})(a_{2n}\blacktriangleleft) & \longrightarrow & (u \overline a_i)(a_{2n}\blacktriangleleft),
\\[.08cm] & (w \overline a_i a_{i-1})(a_{1}\blacktriangleleft) & \longrightarrow & (u \overline a_i)(a_{1}\blacktriangleleft),
\\[.08cm] & (w \overline a_i a_{i-1})(a_{2}\blacktriangleleft) & \longrightarrow & (u \overline a_i)(a_{2}\blacktriangleleft),
\\[.03cm]&& \vdots
\\[.08cm] & (w \overline a_i a_{i-1})(a_{j-1}\blacktriangleleft) & \longrightarrow & (u \overline a_i)(a_{j-1}\blacktriangleleft).
\end{tabular}}

\bigskip
Since $a_j = \overline a_{i-1}$, the first column is ordered with respect to $T_1$. On the other hand, $a_j = \overline a_{i-1}$ implies that the partition of $(u\overline a_i)(a_{i-1} \blacktriangleleft)$ with respect to $T_2$ ends with $(u\overline a_i)(a_{i-1}a_i \blacktriangleleft)$. Then, the second column of this table is ordered with respect to $T_2$. Hence, if $(w \overline a_{i}a_{i-1}) \mathfrak{e}' <_{T_1} (w \overline a_{i}a_{i-1}) \mathfrak{f}'$ then $(u \overline a_{i}) \mathfrak{e}'' <_{T_2} (u \overline a_{i}) \mathfrak{f}''$.\medskip

If $i-1 < j$, then 
$$\begin{array}{ll}
b_{[1{\uparrow}(i-1)]} & = a_{[1{\uparrow}(i-1)]}\\
b_{[i{\uparrow}(j-2)]} & = a_{[(i+1){\uparrow}(j-1)]}\\
b_{j-1} & = a_i\\
b_{[j{\uparrow}2n]} & = a_{[j{\uparrow}2n]}
\end{array}$$
The partition with respect to $a_{[1{\uparrow}2n]}$ of $(\overline a_{j}\blacktriangleleft) = (a_{i-1}\blacktriangleleft)$ is $(a_{j+1}\blacktriangleleft),$ \linebreak $(a_{j+2}\blacktriangleleft),\ldots, (a_{2n}\blacktriangleleft)$,$(a_{1}\blacktriangleleft),(a_{2}\blacktriangleleft), \ldots, (a_{i-1}\blacktriangleleft),(a_{i}\blacktriangleleft),(a_{i+1}\blacktriangleleft), \ldots ,(a_{j-1}\blacktriangleleft)$. The partition with respect to $b_{[1{\uparrow}2n]}$ of $(\overline a_{i}\blacktriangleleft)$ is $(a_{j}\blacktriangleleft),(a_{j+1}\blacktriangleleft), \ldots ,(a_{2n}\blacktriangleleft)$, $(a_{1}\blacktriangleleft),(a_{2}\blacktriangleleft), \ldots ,(a_{i-1}\blacktriangleleft),(a_{i+1}\blacktriangleleft),(a_{i+2}\blacktriangleleft), \ldots ,(a_{j-1}\blacktriangleleft)$. By the third table,
\medskip

\centerline{
\begin{tabular}
{
>{$}l<{$}
@{\hskip.5cm}
>{$}l<{$}
@{\hskip.5cm}
>{$}c<{$}
@{\hskip.5cm}
>{$}l<{$}
}
\setlength\extrarowheight{60pt}
&  & (a_i \mapsto a_{i-1}a_i) & 
\\[.08cm]& (w \overline a_i a_{i-1})(a_{j+1}\blacktriangleleft) & \longrightarrow & (u \overline a_i)(a_{j+1}\blacktriangleleft),
\\[.08cm]& (w \overline a_i a_{i-1})(a_{j+2}\blacktriangleleft) & \longrightarrow & (u \overline a_i)(a_{j+2}\blacktriangleleft),
\\[.03cm]&& \vdots
\\[.08cm] & (w \overline a_i a_{i-1})(a_{2n} \blacktriangleleft) & \longrightarrow & (u \overline a_i)(a_{2n}\blacktriangleleft),
\\[.08cm] & (w \overline a_i a_{i-1})(a_1 \blacktriangleleft) & \longrightarrow & (u \overline a_i)(a_1 \blacktriangleleft),
\\[.08cm] & (w \overline a_i a_{i-1})(a_{2}\blacktriangleleft) & \longrightarrow & (u \overline a_i)(a_{2}\blacktriangleleft),
\\[.03cm]&& \vdots
\\[.08cm] & (w \overline a_i a_{i-1})(a_{i-2}\blacktriangleleft) & \longrightarrow & (u \overline a_i)(a_{i-2}\blacktriangleleft),
\\[.08cm] & (w \overline a_i a_{i-1})(a_{i-1}\blacktriangleleft) & \longrightarrow & (u \overline a_i)[(a_{i-1}\blacktriangleleft)-(a_{i-1}a_i\blacktriangleleft)],
\\[.08cm] & (w \overline a_i a_{i-1})(a_{i}\blacktriangleleft) & \longrightarrow & (u \overline a_i)(a_{i-1}a_i\blacktriangleleft),
\\[.08cm] & (w \overline a_i a_{i-1})(a_{i+1}\blacktriangleleft) & \longrightarrow & (u \overline a_i)(a_{i+1}\blacktriangleleft),
\\[.03cm]&& \vdots
\\[.08cm] & (w \overline a_i a_{i-1})(a_{j-1}\blacktriangleleft) & \longrightarrow & (u \overline a_i)(a_{j-1}\blacktriangleleft).
\end{tabular}}

\bigskip
Since $a_j = \overline a_{i-1}$, the first column is ordered with respect to $T_1$. On the other hand, $a_j = \overline a_{i-1}$ implies that the partition of $(u\overline a_i)(a_{i-1} \blacktriangleleft)$ with respect to $T_2$ ends with $(u\overline a_i)(a_{i-1}a_i \blacktriangleleft)$. Then, the second column of this table is ordered with respect to $T_2$. Hence, if $(w \overline a_{i}a_{i-1}) \mathfrak{e}' <_{T_1} (w \overline a_{i}a_{i-1}) \mathfrak{f}'$ then $(u \overline a_{i}) \mathfrak{e}'' <_{T_2} (u \overline a_{i}) \mathfrak{f}''$.\medskip

For every row of the first table, there is a case which needs to be considered. Similarly, in all these cases, it can be shown that if $\mathfrak{e} <_{T_1} \mathfrak{f}$, then $\mathfrak{e}^{(a_i \mapsto a_{i-1}a_i)} <_{T_2} \mathfrak{f}^{(a_i \mapsto a_{i-1}a_i)}$.\medskip

The case $\phi=(\overline a_i\mapsto \overline a_{i}\overline a_{i+1})$ for some $i\in [1{\uparrow}(2n-1)]$, $a_i \neq \overline a_{i+1}$, is similar.
\end{proof}

\end{review}


\section{$t$-squarefreeness}
\label{sec:free}

Recall $2g + p = n$ and $\Sigma_{g,1,p}$ is the free group on $x_{[1{\uparrow}g]}\vee y_{[1{\uparrow}g]}\vee t_{[1{\uparrow} p]}$.\medskip

The following definition extends Definition~\ref{def:t-lliure} to $\Sigma_{g,1,p} \cup \partial \Sigma_{g,1,p}$.

\begin{definition}\label{def:t-sqfree}
An element of $\Sigma_{g,1,p} \cup \partial \Sigma_{g,1,p}$ is said to be {\it $t$-squarefree}
if, in its reduced expression, no two consecutive terms in $t_{[1{\uparrow}p]}\vee \overline t_{[1{\uparrow}p]}$ are equal.
\end{definition}

\begin{notation}
In the standard surface word set, we denote $$\overline z_1 = \Pi_{i\in {[1{\uparrow}g]}} [x_i, y_i] \Pi t_{[1{\uparrow}p]}\quad \textrm{ and } \quad z_1 = \Pi \overline t_{[p{\downarrow}1]} \Pi_{i\in {[g{\downarrow}1]}} [y_i, x_i].$$

From the last comment of Definition~\ref{def:ends}, the smallest element of $(\partial \Sigma_{g,1,p}, <)$ is $\overline z_1^{\,\infty}$ and the largest element of $(\partial \Sigma_{g,1,p}, <)$ is $z_1^\infty$. We denote by $\min(\partial \Sigma_{g,1,p}) = \overline z_1^{\,\infty}$ and $\max(\partial \Sigma_{g,1,p}) = z_1^\infty$ these facts.\medskip

Given two ends $\mathfrak{e},\mathfrak{f} \in \partial \Sigma_{g,1,p}$, we write $$[\mathfrak{e} {\uparrow} \mathfrak{f}] \coloneq \{\mathfrak{g}  \in \partial \Sigma_{g,1,p} \mid \mathfrak{e} \leq \mathfrak{g} \leq \mathfrak{f}\}.$$
\end{notation}

\begin{lemma}\label{lemma_rep} Let $k_0\in [1{\uparrow} p],\, w\in \Sigma_{g,1,p}-(\star t_{k_0})-(\star \overline t_{k_0})$ and $i_0\in [1{\uparrow} g]$. Then, in $(\partial \Sigma_{g,1,p},\leq)$, the following hold:
\begin{enumerate}[(i).]
\item\label{lemma_rep_1} $wt_{k_0}\overline{w}(\overline{z}_1^{\infty})\leq wt_{k_0}((\Pi t_{[k_0{\uparrow} p]}\Pi_{i\in [1{\uparrow} g]} [x_i,y_i]\Pi t_{[1{\uparrow} (k_0-1)]})^\infty)=\min(wt_{k_0}t_{k_0}\blacktriangleleft)$;
\item\label{lemma_rep_2} $\max(wt_{k_0}t_{k_0}\blacktriangleleft)<\min(w\overline{t}_{k_0}\overline{t}_{k_0}\blacktriangleleft)$;
\item\label{lemma_rep_3} $\max(w\overline{t}_{k_0}\overline{t}_{k_0}\blacktriangleleft)=w\overline{t}_{k_0}((\Pi \overline{t}_{[k_0{\downarrow} 1]}\Pi_{i\in [g{\downarrow} 1]} [y_i,x_i]\Pi \overline{t}_{[p{\downarrow} (k_0+1)]})^\infty)\leq w\overline{t}_{k_0}\overline{w}(z_1^\infty)$;
\item\label{lemma_rep_4} $(wt_{k_0}t_{k_0}\blacktriangleleft)\cup(w\overline{t}_{k_0}\overline{t}_{k_0}\blacktriangleleft)\subseteq [wt_{k_0}\overline{w}(\overline{z}_1^\infty){\uparrow} w\overline{t}_{k_0}\overline{w}(z_1^\infty)]$;
\item\label{lemma_rep_5} If $2g+p\geq 3$, then one of the following holds:
\begin{enumerate}[(a)]
\item $\overline{t}_p(\overline{z}_1^\infty)>w\overline{t}_{k_0}\overline{w}(z_1^\infty)$;
\item $\overline{t}_p(\overline{z}_1^\infty)<wt_{k_0}\overline{w}(\overline{z}_1^\infty)$;
\end{enumerate}
and, hence, $\overline{t}_p(\overline{z}_1^\infty)\notin [wt_{k_0}\overline{w}(\overline{z}_1^\infty){\uparrow} w\overline{t}_{k_0}\overline{w}(z_1^{\infty})]$;
\item\label{lemma_rep_6} If $a\in\{x_{i_0},\,\overline{x}_{i_0},\,y_{i_0},\overline{y}_{i_0}\}$, then one of the following holds:
\begin{enumerate}[(a).]
\item $a(z_1^\infty)>w\overline{t}_{k_0}\overline{w}(z_1^\infty)$;
\item $a(z_1^\infty)<wt_{k_0}\overline{w}(\overline{z}_1^\infty)$;
\end{enumerate}
and, hence,  $a(z_1^\infty)\notin [wt_{k_0}\overline{w}(\overline{z}_1^\infty){\uparrow} w\overline{t}_{k_0}\overline{w}(z_1^{\infty})]$.
\end{enumerate}
\end{lemma}

\begin{proof} 
Recall $<$ is the ordering with respect to sequence the $$(\overline{x}_1,y_1,x_1,\overline{y}_1,\overline{x}_2,y_2,x_2,\overline{y}_2,\cdots ,\overline{x}_g,y_g,x_g,\overline{y}_g,t_1,\overline{t}_1,t_2,\overline{t}_2,\cdots,t_p,\overline{t}_p).$$

(i). It is straightforward to see that 
$$wt_{k_0}((\Pi t_{[k_0{\uparrow} p]}\Pi_{i\in [1{\uparrow}g]} [x_i,y_i]\Pi t_{[1{\uparrow} (k_0-1)]})^\infty)=\min(wt_{k_0}t_{k_0} \blacktriangleleft).$$
Let $a\in X\vee \overline{X}$ be such that $\overline{w}((\Pi_{i\in [1{\uparrow}g]} [x_i,y_i]\Pi t_{[1{\uparrow} p]})^\infty)\in (a\blacktriangleleft)$. Note $a\neq \overline{t}_{k_0}$.
 
If $a\neq t_{k_0}$, then $(wt_{k_0}a\blacktriangleleft) < (wt_{k_0}t_{k_0}\blacktriangleleft)$, and we have
$$wt_{k_0}\overline{w}(\overline{z}_1^{\infty})=wt_{k_0}\overline{w}((\Pi_{i\in [1{\uparrow}g]} [x_i,y_i]\Pi t_{[1{\uparrow} p]})^{\infty})< \min(wt_{k_0}t_{k_0}).$$
If $a=t_{k_0}$, then $\overline{w}$ is completely canceled in $\overline{w}((\Pi_{i\in [1{\uparrow}g]} [x_i,y_i]\Pi t_{[1{\uparrow} p]})^\infty)$, and, moreover,
\begin{align*}
wt_{k_0}\overline{w}(\overline{z}_1^\infty) & =wt_{k_0}\overline{w}((\Pi_{i\in [1{\uparrow}g]} [x_i,y_i]\Pi t_{[1{\uparrow} p]})^\infty)\\
& =wt_{k_0}((\Pi t_{[k_0{\uparrow} p]}\Pi_{i\in [1{\uparrow}g]} [x_i,y_i]\Pi t_{[1{\uparrow} k_0-1]})^\infty)\\
& =\min(wt_{k_0}t_{k_0} \blacktriangleleft).
\end{align*}

(ii). It is clear.\medskip

(iii). It is straightforward to see that 
$$w\overline{t}_{k_0}((\Pi \overline{t}_{[k_0{\downarrow} 1]}\Pi_{i\in [g{\downarrow} 1]} [y_i,x_i]\Pi \overline{t}_{[p{\downarrow} (k_0-1)]})^\infty)=\max(w\overline{t}_{k_0}\overline{t}_{k_0} \blacktriangleleft).$$
Let $a\in X\vee \overline{X}$ be such that $\overline{w}((\Pi \overline{t}_{[p{\downarrow} 1]}\Pi_{i\in [g{\downarrow} 1]} [y_i,x_i])^\infty)\in (a\blacktriangleleft)$. Note $a\neq t_{k_0}$.

If $a\neq \overline{t}_{k_0}$, then $(w\overline{t}_{k_0}\overline{t}_{k_0}\blacktriangleleft)< (w\overline{t}_{k_0}a\blacktriangleleft),$ and we have
$$\max(w\overline{t}_{k_0}\overline{t}_{k_0}\blacktriangleleft)< w\overline{t}_{k_0}\overline{w}((\Pi \overline{t}_{[p{\downarrow} 1]}\Pi_{i\in [g{\downarrow} 1]} [y_i,x_i])^\infty)= w\overline{t}_{k_0}\overline{w}(z_1^\infty).$$
If $a=\overline{t}_{k_0}$, then $\overline{w}$ is completely canceled in $\overline{w}((\Pi  \overline{t}_{[p{\downarrow} 1]}\Pi_{i\in [g{\downarrow} 1]} [y_i,x_i])^\infty)$, and, moreover,
\begin{align*}
w\overline{t}_{k_0}\overline{w}(z_1^\infty) & =w\overline{t}_{k_0}\overline{w}((\Pi \overline{t}_{[p{\downarrow} 1]}\Pi_{i\in [g{\downarrow} 1]} [y_i,x_i])^\infty)\\
& =w\overline{t}_{k_0}((\Pi \overline{t}_{[k_0{\downarrow} 1]}\Pi_{i\in [g{\downarrow} 1]} [y_i,x_i]\Pi \overline{t}_{[p{\downarrow} (k_0+1)]})^\infty)\\
& =\max(w\overline{t}_{k_0}\overline{t}_{k_0} \blacktriangleleft).
\end{align*}\medskip

(iv). Follows from (i)-(iii).\medskip

(v). By (i)-(iii),
$$wt_{k_0}\overline{w}((\Pi_{i\in [1{\uparrow}g]} [x_i,y_i]\Pi t_{[1{\uparrow} p]})^\infty)< w\overline{t}_{k_0}\overline{w}((\Pi \overline{t}_{[p{\downarrow} 1]}\Pi_{i\in [g{\downarrow}1]} [y_i,x_i])^\infty).$$

\textbf{Case 1.} $w=1$. Since $(\overline{t}_p\overline{x}_1\blacktriangleleft)\cup (\overline{t}_pt_1 \blacktriangleleft) >(\overline{t}_{k_0}\overline{t}_{p}\blacktriangleleft)$, we see
$$\overline{t}_p(\overline{z}_1^\infty)=\overline{t}_p((\Pi_{i\in [1{\uparrow}g]} [x_i,y_i]\Pi t_{[1{\uparrow} p]})^\infty) >\overline{t}_{k_0}((\Pi \overline{t}_{[p{\downarrow} 1]}\Pi_{i\in [g{\downarrow} 1]} [y_i,x_i])^\infty)=\overline{t}_{k_0}(z_1^\infty).$$
Thus, (a) holds.\medskip

\textbf{Case 2.} $w\notin (\overline{t}_p\star)\cup \{1\}$. Since $(\overline{t}_p\blacktriangleleft) >(w\overline{t}_{k_0}\blacktriangleleft)$, we see
\[\begin{split}
\overline{t}_p(\overline{z}_1^\infty)= \overline{t}_p((\Pi_{i\in [1{\uparrow}g]} [x_i,y_i]\Pi t_{[1{\uparrow} p]})^\infty)\quad\quad\quad\quad\quad\quad\quad\quad\,\\
\quad\quad > w\overline{t}_{k_0}\overline{w}((\Pi \overline{t}_{[p{\downarrow} 1]}\Pi_{i\in [g{\downarrow} 1]} [y_i,x_i])^\infty)=w\overline{t}_{k_0}\overline{w}(z_1^\infty).
\end{split}
\]
Thus, (a) holds.\medskip

\textbf{Case 3.} $w\in (\overline{t}_p\overline{t}_p\star)$. Since $(\overline{t}_p\overline{x}_1\blacktriangleleft)\cup (\overline{t}_pt_1 \blacktriangleleft) >(w\overline{t}_{k_0}\blacktriangleleft)$, we see
\[\begin{split}
\overline{t}_p(\overline{z}_1^\infty)=\overline{t}_p((\Pi_{i\in [1{\uparrow}g]} [x_i,y_i]\Pi t_{[1{\uparrow} p]})^\infty)\quad\quad\quad\quad\quad\quad\quad\quad\,\\
> w\overline{t}_{k_0}\overline{w}((\Pi \overline{t}_{[p{\downarrow} 1]}\Pi_{i\in [g{\downarrow} 1]} [y_i,x_i])^\infty)=w\overline{t}_{k_0}\overline{w}(z_1^\infty).
\end{split}
\]
Thus, (a) holds.\medskip

\textbf{Case 4.} $w\in (\overline{t}_p\star)-(\overline{t}_p\overline{t}_p\star)$.

Here, $$w t_{k_0}\overline{w}(\overline z_1^\infty) = wt_{k_0}\overline{w}((\Pi_{[1{\uparrow}g]}[x_i,y_i] \Pi t_{[1{\uparrow} p]})^{\infty})\in (wt_{k_0}\blacktriangleleft)\subset (\overline{t}_p\blacktriangleleft)- (\overline{t}_p\overline{t}_p\blacktriangleleft).$$ Hence,
\[
\begin{split}
\overline{t}_p((\Pi_{i\in [1{\uparrow}g]} [x_i,y_i]\Pi t_{[1{\uparrow} p]})^\infty)=\min((\overline{t}_p\blacktriangleleft)-(\overline{t}_p\overline{t}_p\blacktriangleleft))\quad\quad\quad\,\,\,\,\,\\
 \leq wt_{k_0}\overline{w}((\Pi_{i\in [1{\uparrow}g]} [x_i,y_i]\Pi t_{[1{\uparrow} p]})^\infty).
\end{split}
\]
To prove (b) holds, it remains to show that $$\overline{t}_p((\Pi_{i\in [1{\uparrow}g]} [x_i,y_i]\Pi t_{[1{\uparrow} p]})^\infty)\neq wt_{k_0}\overline{w}((\Pi_{i\in [1{\uparrow}g]} [x_i,y_i]\Pi t_{[1{\uparrow} p]})^\infty),$$
that is, $(\Pi_{i\in [1{\uparrow}g]} [x_i,y_i]\Pi t_{[1{\uparrow} p]})^\infty\neq t_pwt_{k_0}\overline{w}((\Pi_{i\in [1{\uparrow}g]} [x_i,y_i]\Pi t_{[1{\uparrow} p]})^\infty)$, that is, \newline
$t_pwt_{k_0}\overline{w}\notin \gen{\Pi_{i\in [1{\uparrow}g]} [x_i,y_i]\Pi t_{[1{\uparrow} p]}}.$ We can write $w=\overline{t}_pu$ where $u\notin (t_p\star)$. Then \newline
$t_pwt_{k_0}\overline{w}=ut_{k_0}\overline{u}t_p$, in normal form. Thus it suffices to show $$ut_{k_0}\overline{u}t_p\notin \gen{\Pi_{i\in [1{\uparrow} g]} [x_i,y_i]\Pi t_{[1{\uparrow} p]}}.$$

If $u=1$, then $ut_{k_0}\overline{u}t_p\notin \gen{\Pi_{i\in [1{\uparrow}g]} [x_i,y_i]\Pi t_{[1{\uparrow} p]}}$, since $2g+p\geq 3$.\medskip

If $u\neq 1$, then $ut_{k_0}\overline{u}t_p\notin \gen{\Pi_{i\in [1{\uparrow}g]} [x_i,y_i]\Pi t_{[1{\uparrow} p]}}$, since $ut_{k_0}\overline{u}t_p$ does not lie in the submonoid of $\Sigma_{g,1,p}$ generated by $\Pi_{i\in [1{\uparrow}g]} [x_i,y_i]\Pi t_{[1{\uparrow} p]}$, nor in the submonoid generated by $\Pi \overline{t}_{[p{\downarrow} 1]}\Pi_{i\in [g{\downarrow} 1]} [y_i,x_i]$.\medskip

In all four cases (v) holds.\medskip

(vi). Let $a\in\{x_{i_0},\,\overline{x}_{i_0},\,y_{i_0},\,\overline{y}_{i_0}\}$.\medskip

\textbf{Case 1.} $w=1$. Since $(a\blacktriangleleft) <(t_{k_0}\blacktriangleleft)$, we see
$$a(z_1^{\infty}) = a((\Pi \overline{t}_{[p{\downarrow} 1]}\Pi_{i\in [g{\downarrow} 1]} [y_i,x_i])^\infty)< t_{k_0}((\Pi_{i\in [1{\uparrow}g]} [x_i,y_i]\Pi t_{[1{\uparrow} p]})^\infty) = t_{k_0} (\overline{z}_1^{\infty}).$$

Thus, (b) holds.\medskip

\textbf{Case 2.} $w\notin (a\star)\cup\{1\}$.\medskip

If $(a\blacktriangleleft)>(w\blacktriangleleft)$, then $(a\blacktriangleleft) >(w\blacktriangleleft)\supset(w\overline{t}_{k_0}\blacktriangleleft)$ and 
\[
\begin{split}
a(z_1^{\infty}) = a((\Pi \overline{t}_{[p{\downarrow} 1]}\Pi_{i\in [g{\downarrow} 1]} [y_i,x_i])^\infty)\quad\quad\quad\quad\quad\quad\quad\\
> w\overline{t}_{k_0}\overline{w}((\Pi \overline{t}_{[p{\downarrow} 1]}\Pi_{i\in [g{\downarrow} 1]} [y_i,x_i])^\infty) = w\overline{t}_{k_0}\overline{w} (z_1^{\infty}).
\end{split}
\]
Thus, (a) holds.\medskip

If $(a\blacktriangleleft)<(w\blacktriangleleft)$, then $(a\blacktriangleleft) <(w\blacktriangleleft)\supset(wt_{k_0}\blacktriangleleft)$ and 
\[
\begin{split}
a(z_1^{\infty}) = a((\Pi \overline{t}_{[p{\downarrow} 1]}\Pi_{i\in [g{\downarrow} 1]} [y_i,x_i])^\infty)\quad\quad\quad\quad\quad\quad\quad\quad\,\,\\
< wt_{k_0}\overline{w}((\Pi_{i\in [1{\uparrow}g]} [x_i,y_i]\Pi t_{[1{\uparrow} p]})^\infty) = wt_{k_0}\overline{w} (\overline{z}_1^{\infty}).
\end{split}\]
Thus, (b) holds.\medskip

\textbf{Case 3.} $w\in (a\overline{t}_p\star)$.

Since $a((\Pi \overline{t}_{[p{\downarrow} 1]}\Pi_{i\in [g{\downarrow} 1]} [y_i,x_i])^\infty)=\max(a\overline{t}_p\blacktriangleleft)$, we see 
\[\begin{split}
a(z_1^{\infty}) = a((\Pi \overline{t}_{[p{\downarrow} 1]}\Pi_{i\in [g{\downarrow} 1]} [y_i,x_i])^\infty)\quad\quad\quad\quad\quad\quad\quad\quad\,\,\,\\
\geq w\overline{t}_{k_0}\overline{w}((\Pi \overline{t}_{[p{\downarrow} 1]}\Pi_{i\in [g{\downarrow} 1]} [y_i,x_i])^\infty) =  w\overline{t}_{k_0}\overline{w} (z_1^{\infty}).
\end{split}
\]

To prove (a) holds, it remains to show that $$a((\Pi \overline{t}_{[p{\downarrow} 1]}\Pi_{i\in [g{\downarrow} 1]} [y_i,x_i])^\infty)\neq w\overline{t}_{k_0}\overline{w}((\Pi \overline{t}_{[p{\downarrow} 1]}\Pi_{i\in [g{\downarrow} 1]} [y_i,x_i])^\infty),$$ that is, $((\Pi \overline{t}_{[p{\downarrow} 1]}\Pi_{i\in [g{\downarrow} 1]} [y_i,x_i])^\infty)\neq \overline{a}w\overline{t}_{k_0}\overline{w}((\Pi \overline{t}_{[p{\downarrow} 1]}\Pi_{i\in [g{\downarrow} 1]} [y_i,x_i])^\infty)$, that is\newline 
$\overline{a}w\overline{t}_{k_0}\overline{w}\notin\gen{\Pi \overline{t}_{[p{\downarrow} 1]}\Pi_{i\in [g{\downarrow} 1]} [y_i,x_i]n}$. We can write $w=a\overline{t}_pu$ where $u\notin (t_p\star)$. Then $\overline{a}w\overline{t}_{k_0}\overline{w}=\overline{t}_pu\overline{t}_{k_0}\overline{u}t_p\overline{a}$, in normal form. Thus it suffices to show that $$\overline{t}_pu\overline{t}_{k_0}\overline{u}t_p\overline{a}\notin\gen{\Pi \overline{t}_{[p{\downarrow} 1]}\Pi_{i\in [g{\downarrow} 1]} [y_i,x_i]},$$ which is clear since $\overline{t}_pu\overline{t}_{k_0}\overline{u}t_p\overline{a}$ does not lie in the submonoid of $\Sigma_{g,1,p}$ generated by
$\Pi \overline{t}_{[p{\downarrow} 1]}\Pi_{i\in [g{\downarrow} 1]} [y_i,x_i]$, nor in the submonoid generated by $\Pi_{i\in [1{\uparrow}g]} [x_i,y_i]\Pi t_{[1{\uparrow} p]}$.\medskip

\textbf{Case 4.} $w\in (a\star)-(a\overline{t}_p\star),\, \abs{w}\geq 2$.

If $(a\overline{t}_p\blacktriangleleft)>(w\blacktriangleleft)$, then $(a\overline{t}_p\blacktriangleleft) >(w\blacktriangleleft)\supset(w\overline{t}_{k_0}\blacktriangleleft)$ and
\[
\begin{split}
a(z_1^{\infty}) = a((\Pi \overline{t}_{[p{\downarrow} 1]}\Pi_{i\in [g{\downarrow} 1]} [y_i,x_i])^\infty)\quad\quad\quad\quad\quad\quad\quad\\
> w\overline{t}_{k_0}\overline{w}((\Pi \overline{t}_{[p{\downarrow} 1]}\Pi_{i\in [g{\downarrow} 1]} [y_i,x_i])^\infty) = w\overline{t}_{k_0}\overline{w}(z_1^{\infty}).
\end{split}
\]
Thus, (a) holds.\medskip

If $(a\overline{t}_p\blacktriangleleft)<(w\blacktriangleleft)$, then $(a\overline{t}_p\blacktriangleleft) <(w\blacktriangleleft)\supset(wt_{k_0}\blacktriangleleft)$ and 
\[
\begin{split}
a(z_1^{\infty}) = a((\Pi \overline{t}_{[p{\downarrow} 1]}\Pi_{i\in [g{\downarrow} 1]} [y_i,x_i])^\infty)\quad\quad\quad\quad\quad\quad\quad\quad\,\,\,\\
< wt_{k_0}\overline{w}((\Pi_{i\in [1{\uparrow}g]} [x_i,y_i]\Pi t_{[1{\uparrow} p]})^\infty)  = wt_{k_0}\overline{w}(\overline{z}_1^{\infty}).
\end{split}
\]
Thus, (b) holds.\medskip

\textbf{Case 5.} $w=a$.

Since $a(z_1^\infty) = \max(a\overline t_p\blacktriangleleft),\, (a\overline t_p\blacktriangleleft) \supset (a\overline t_p \overline y_g \overline x_g \blacktriangleleft )$ and $(a\overline t_p \overline y_g \overline x_g \blacktriangleleft) > (a\overline t_{k_0} \overline a \overline t_p \blacktriangleleft)$, we see 
\[
\begin{split}
a(z_1^{\infty}) = a((\Pi \overline{t}_{[p{\downarrow} 1]}\Pi_{i\in [g{\downarrow} 1]} [y_i,x_i])^\infty)\quad\quad\quad\quad\quad\quad\quad\\
> w\overline{t}_{k_0}\overline{w}((\Pi \overline{t}_{[p{\downarrow} 1]}\Pi_{i\in [g{\downarrow} 1]} [y_i,x_i])^\infty) = w\overline{t}_{k_0}\overline{w}(z_1^{\infty}).
\end{split}
\]
Thus, (a) holds.\medskip

In all five cases (vi) holds.
\end{proof}

\begin{theorem}~\label{th:f_sq} If $2g+p \geq 3$ then, for each $\phi\in \AM_{g,1,p}$,
\begin{enumerate}[(i).]
\item\label{th:f_sq_1} $\overline{t}_p^\phi(\overline{z}_1^\infty)$ is a $t$\d1squarefree end,
\item\label{th:f_sq_2} for every $i_0\in [1{\uparrow}g]$ and every $a\in\{x_{i_0},\overline{x}_{i_0},y_{i_0},\overline{y}_{i_0}\}$, $a^\phi(z_1^\infty)$ is a $t$\d1squarefree end.
\end{enumerate}
\end{theorem}

\begin{proof}
(i).
Recall $\overline{z}_1=\Pi_{i\in [1{\uparrow}g]} [x_i,y_i]\Pi t_{[1{\uparrow} p]}$ and $z_1=\Pi \overline{t}_{[p{\downarrow} 1]}\Pi_{i\in [g{\downarrow} 1]} [y_i,x_i]$. Let us $\cup[t]_{[1{\uparrow} p]}$ denote $\bigcup_{k\in[1{\uparrow} p]}[t_k]$. By Lemma~\ref{lemma_rep}\eqref{lemma_rep_5}, $\overline{t}_p(\overline{z}_1^\infty)$ does not lie in $$\bigcup_{u\in\cup[t]_{[1{\uparrow} p]}}[(u(\overline{z}_1^\infty)){\uparrow} (\overline{u}(z_1^{\infty}))](=\bigcup_{k=1}^p\bigcup_{w\in \Sigma_{g,1,p}-(\star t_k)-(\star \overline{t}_k)}[ (wt_k\overline{w}(\overline{z}_1^\infty)){\uparrow} (w\overline{t}_k\overline{w}(z_1^\infty))]).$$

Notice that $\phi$ permutes the elements of each of the following sets:\medskip

$$\cup[t]_{[1{\uparrow} p]};\quad\quad \{\overline{z}_1^\infty\};\quad\quad \{z_1^\infty\};\quad\quad \textrm{and}\quad\quad \bigcup_{u\in[t]_{[1{\uparrow} p]}}[u(\overline{z}_1^\infty){\uparrow} \overline{u}(z_1^\infty)].$$

Hence, $(\overline{t}_p(\overline{z}_1^\infty))^\phi$ does not lie in $\bigcup_{u\in [t]_{[1{\uparrow} p]}}[u(\overline{z}_1^\infty){\uparrow} \overline{u}(z_1^\infty)]$. By Lemma \ref{lemma_rep}\eqref{lemma_rep_4}, $$\bigcup_{u\in[t]_{[1{\uparrow} p]}}[u(\overline{z}_1^\infty){\uparrow} \overline{u}(z_1^\infty)]\supseteq \bigcup_{k=1}^p\bigcup_{w\in\Sigma_{g,1,p}-(\star t_k)-(\star \overline{t}_k)}((wt_kt_k\blacktriangleleft)\cup(w\overline{t}_k\overline{t}_k\blacktriangleleft)).$$

Hence, $(\overline{t}_p(\overline{z}_1^\infty))^\phi$ does not lie in the right-hand side set either, and, hence, $(\overline{t}_p(\overline{z}_1^\infty))^\phi$ is a $t$\d1squarefree end. Since $(\overline{t}_p(\overline{z}_1^\infty))^\phi=\overline{t}_p^\phi(\overline{z}_1^\infty)$, the desired result holds.
\medskip

(ii). The same proof as (i) using Lemma \ref{lemma_rep}\eqref{lemma_rep_6} instead of Lemma \ref{lemma_rep}\eqref{lemma_rep_5}.
\end{proof}

\begin{proof}{\it (of Theorem~\ref{thm:t-sqfree})} Recall \eqref{eq:trena}.
$\AM_{0,1,2}=\gen{\sigma_1}$, and $$t_2^{\AM_{0,1,2}}=\{t_2^{\sigma_1^{2m}},t_2^{\sigma_1^{2m+1}}\mid m\in \integers\}=\{t_2^{(t_1t_2)^m},t_1^{(t_1t_2)^m}\mid m\in \integers\}$$

Thus, every element of $t_2^{\AM_{0,1,2}}$ is $t$\d1squarefree.\medskip

Suppose, now, $2g+p\geq 3$. Let $i_0\in[1{\uparrow}g]$ and $a\in\{x_{i_0},y_{i_0}\}$. By Theorem~\ref{th:f_sq}\eqref{th:f_sq_2}, $a^\phi(z_1^\infty)=a^\phi((\Pi \overline{t}_{[p{\downarrow} 1]}\Pi [x,y]_{[g{\downarrow} 1]})^\infty)	$ is a $t$\d1squarefree end. Hence, either $a^\phi$ is $t$\d1squarefree or $a^\phi=ut_kt_kv$ in normal form, and $t_kv$ is canceled in $a^\phi(z_1^\infty) = ut_kt_kv(z_1^\infty)$; moreover $ut_k,t_kv$ are $t$\d1squarefree. By Theorem~\ref{th:f_sq}\eqref{th:f_sq_2}, $$\overline{a}^\phi(z_1^\infty) = \overline{a}^\phi((\Pi \overline{t}_{[p{\downarrow} 1]}\Pi [x,y]_{[g{\downarrow} 1]})^\infty)$$ is a $t$\d1squarefree end. Hence, $\overline{a}^\phi\neq \overline{v}\overline{t}_k\overline{t}_k\overline{u}$.\medskip

Since $\phi$ permutes $\bigcup_{k\in[1{\uparrow} p]}[t_k]$, we can write $\overline{t}_p^\phi=\overline{t}^{w_p}_{p^\pi}$, where $\pi$ is a permutation of $[1{\uparrow}p]$ and $w_p\in \Sigma_{g,1,p} - (t_{p^\pi} \star) - (\overline t_{p^\pi} \star)$. It is not difficult to see that
$$\overline{t}_p^\phi(\overline{z}_1^{\infty})=\overline{w}_p\overline{t}_{p^\pi}w_p((\Pi_{i\in [1{\uparrow}g]} [x_i,y_i]\Pi t_{[1{\uparrow} p]})^\infty)\in (\overline{w}_p\blacktriangleleft).$$

By Theorem~\ref{th:f_sq}\eqref{th:f_sq_1}, $\overline{t}_p^\phi(\overline{z}_1^\infty)$ is a $t$\d1squarefree end. Hence, $\overline{w}_p$ is $t$\d1squarefree.\medskip

Since $\overline{w}_p$ is $t$\d1squarefree, $\overline{t}_p^\phi=\overline{w}_p\overline{t}_pw_p$ is also $t$\d1squarefree. Hence, $t_p^\phi$ is $t$\d1squarefree. \medskip

Suppose, now, $2g+p\geq 2$. Let $k\in [1{\uparrow}p]$. Since $t_k$ is in the $\AM_{g,1,p}$\d1orbit of $t_p$, $t_k^\phi$ is $t$\d1squarefree for all $\phi \in \AM_{g,1,p}$.
\end{proof}

\noindent{\textbf{{Acknowledgments}}}

\medskip
\footnotesize

The autor is grateful to Warren Dicks and Luis Paris for many interesting observations.

\vskip -.5cm\null

\bibliographystyle{amsplain}

\medskip

\noindent\textsc{Llu\'is Bacardit,\newline
Institut de Mathématiques de Bourgogne\newline
Universit\'e de Bourgogne\newline
UMR 5584 du CNRS, BP 47870\newline
21078 Dijon Cedex\newline
France}

\medskip

\noindent \emph{E-mail address}{:\;\;}\texttt{Lluis.Bacardit@u-bourgogne.fr}

\end{document}